\documentclass[11pt]{amsart}
\usepackage{a4wide,pdfsync,hyperref}
\usepackage{amsmath,amssymb,esint}
\usepackage{enumerate}
\usepackage{color}

\numberwithin{equation}{section}

\allowdisplaybreaks

\newtheorem{theorem}{Theorem}[section]
\newtheorem{lemma}[theorem]{Lemma}

\newtheorem{corollary}[theorem]{Corollary}
\newtheorem{proposition}[theorem]{Proposition}
\newtheorem{remark}[theorem]{Remark}

\newcommand{\beqq}{\begin{eqnarray}}
\newcommand{\enqq}{\end{eqnarray}}

\newcommand{\enn}{\end{equation}}
\newcommand{\bef}{\begin{proof}}
\newcommand{\enf}{\end{proof}}
\let\d=\delta
\let\e=\varepsilon

\let\la=\lambda

\let\f=\frac

\let\om=\omega

\let\D=\Delta

\let\tri=\triangle
\let\na=\nabla

\let\pa=\partial

\def\cS{{\mathcal S}}

\def\Re{\mathbf {Re}}

\def\R{\mathbf R}

\def\no{\noindent}

\def\dv{\mbox{div}}

\def\D{\langle D_x\rangle}
\def\k{\langle k\rangle}

\def\eqdef{\buildrel\hbox{\footnotesize def}\over =}

\newcommand{\beq}{\begin{equation}}
\newcommand{\eeq}{\end{equation}}
\newcommand{\ben}{\begin{eqnarray}}
\newcommand{\een}{\end{eqnarray}}
\newcommand{\beno}{\begin{eqnarray*}}
\newcommand{\eeno}{\end{eqnarray*}}

\begin{document}
\title[Optimal Gevrey stability of hydrostatic approximation]{Optimal Gevrey stability of hydrostatic approximation for the Navier-Stokes equations in a thin domain }

\author[C. Wang]{Chao Wang}
\address{School of Mathematical Sciences\\ Peking University\\ Beijing 100871,China}
\email{wangchao@math.pku.edu.cn}

\author[Y. Wang]{Yuxi Wang}
\address{School of Mathematics, Sichuan University, Chengdu 610064, China}
\email{wangyuxi@scu.edu.cn}

%\author[Z. Zhang]{Zhifei Zhang}
%\address{School of Mathematical Sciences\\ Peking University\\ Beijing 100871,China}
%\email{zfzhang@math.pku.edu.cn}

\maketitle

\begin{abstract}

In this paper, we study the hydrostatic approximation for the Navier-Stokes system in a thin domain.  When the convex initial data with Gevrey regularity of optimal index $\f32$ in $x$ variable and Sobolev regularity in $y$ variable, we justify the limit from the anisotropic Navier-Stokes system to the hydrostatic Navier-Stokes/Prandtl system. Due to our method in the paper is independent of $\e$, by the same argument, we also get the hydrostatic Navier-Stokes/Prandtl system is well-posedness in the optimal Gevrey space. Our results improve the Gevrey index in \cite{GMV, WWZ1} whose Gevrey index is $\f{9}8$.
\end{abstract}

\section{introduction}
\subsection{Presentation the problem and related results}
In this article, we study 2-D incompressible Navier-Stokes equations in a thin domain where the aspect ratio and  the Reynolds number have certain constraints:
\begin{align}\label{eq:NS}
\left\{
\begin{aligned}
&\partial_{t} U+U \cdot \nabla U-\varepsilon^{2} (\pa_x^2+\eta\pa_y^2) U+\nabla P=0,\\
&\dv~ U=0,\\
&U|_{y=0}=U|_{y=\e}=0,
\end{aligned}
\right.
\end{align}
where $t\geq0, (x,y)\in \mathcal{S}^\e=\left\{(x, y) \in \mathbb{T}\times\mathbb{R} : 0<y<\varepsilon\right\}.$
 Here, $U(t,x,y), P(t,x,y)$ stand for the velocity and pressure function respectively and $\eta$ is a positive constant independent of $\varepsilon$. The width of domain $\mathcal{S}^{\e}$ is $\e,$ and the boundary condition in \eqref{eq:NS} corresponds to non-slip condition at the walls $y = 0, \e.$  In addition,
the system is prescribed with the initial data of the form
\begin{equation}\label{def: initial}
\left.U\right|_{t=0}=\left(u_{0}\left(x, \frac{y}{\varepsilon}\right), \varepsilon v_{0}\left(x, \frac{y}{\varepsilon}\right)\right)=U_{0}^{\varepsilon} \quad \text {in}\quad \mathcal{S}^{\varepsilon}.
\end{equation}
This is a classical model with applications to oceanography, meteorology and geophysical flows, where the vertical dimension of the domain is very small compared with the  horizontal  dimension of the domain. 

To study the process $\e\to 0$, we firstly fix the domain independent of $\e$. Here, we rescale the system \eqref{eq:NS} as follows:
\begin{equation*}
U(t, x, y)=\left(u^{\varepsilon}\left(t, x, \frac{y}{\varepsilon}\right), \varepsilon v^{\varepsilon}\left(t, x, \frac{y}{\varepsilon}\right)\right) \quad \text { and } \quad P(t, x, y)=p^{\varepsilon}\left(t, x, \frac{y}{\varepsilon}\right).
\end{equation*}
We put above relations into  \eqref{eq:NS}, then   \eqref{eq:NS} is reduced to a scaled anisotropic Navier-Stokes system:
\begin{align}\label{eq:ANS}
\left\{
\begin{aligned}
&\pa_t u^\e+u^\e\pa_x u^\e+v^\e\pa_y u^\e-\e^2\pa_x^2 u^\e-\eta\pa_y^2 u^\e+\pa_x p^\e=0,\\
&\e^2(\pa_t v^\e+u^\e\pa_x v^\e+v^\e\pa_y v^\e-\e^2\pa_x^2 v^\e-\eta\pa_y^2 v^\e)+\pa_y p^\e=0,\\
&\pa_x u^\e+\pa_y v^\e=0,\\
&(u^\e,v^\e)|_{y=0,1}=0,\\
&(u^\e,v^\e)|_{t=0}=(u_0,v_0) ,
\end{aligned}
\right.
\end{align}
where $(x,y)\in\mathcal{S}=\left\{(x, y) \in \mathbb{T}\times(0,1) \right\}$. 

To simplify the notations, we take $\eta=1$ in this paper and denote $\Delta_\e=\e^2\pa_x^2+\pa_y^2$.

Formally, taking $\e\to 0$ in \eqref{eq:ANS}, we derive the hydrostatic Navier-Stokes/Prandtl system (see \cite{LL, Renardy}):
\begin{align}\label{eq:HyNS}
\left\{
\begin{aligned}
&\pa_t u_p+u_p\pa_x u_p+v_p\pa_y u_p-\eta\pa_y^2 u_p+\pa_x p_p=0\quad \text { in } \mathcal{S} \times (0, \infty),\\
&\pa_y p_p=0 \quad \text { in } \mathcal{S} \times (0, \infty),\\
&\pa_x u_p+\pa_y v_p=0 \quad \text { in } \mathcal{S} \times (0, \infty),\\
&(u_p,v_p)|_{y=0,1}=0,\\
&u_p|_{t=0}=u_0 \quad \text { in }\,\, \mathcal{S}.
\end{aligned}
\right.
\end{align}

The goal in this paper is to justify the limit from the scaled anisotropic Navier-Stokes system  \eqref{eq:ANS} to the hydrostatic Navier-Stokes/Prandtl system \eqref{eq:HyNS}, for a class of convex data in the optimal Gevrey class with index $\gamma=\f32$. 
\medskip

Before presenting the precise statement of the main result in this paper, we recall some results on system \eqref{eq:HyNS}. If $\eta=0$ in the system \eqref{eq:HyNS}, we get the hydrostatic Euler system: 
\begin{align}\label{eq:HyE}
\left\{
\begin{aligned}
&\pa_t u+u\pa_x u+v\pa_y u +\pa_x p=0\quad \text { in } \mathcal{S} \times (0, \infty) ,\\
&\pa_y p=0\quad \text { in } \mathcal{S} \times (0, \infty) ,\\
&\pa_x u+\pa_y v=0 \quad \text { in } \mathcal{S} \times (0, \infty),\\
&(u,v)|_{y=0,1}=0,\\
&u|_{t=0}=u_0 \quad \text { in }\,\, \mathcal{S}.
\end{aligned}
\right.
\end{align}

There are a lot of research on the system \eqref{eq:HyE}, and readers can refer to \cite{Brenier, Brenier1, CINT, Grenier, KTVZ, KMVW, MW1, Renardy, Wong}. Renardy \cite{Renardy} proved the linearization of \eqref{eq:HyE} has a growth  like $e^{|k|t}$ if the initial data is not uniform convexity (or concavity) in variable $y$.  Local well-posedness in the analytic setting was established in \cite{KTVZ}. Under the convexity condition, Masmoudi and Wong \cite{MW1} got the well-posedness of \eqref{eq:HyE} in the Sobolev space.

Next, we recall some results on the well-posedness of  the hydrostatic Navier-Stokes/Prandtl system \eqref{eq:HyNS}. Similar to the classical Prandtl equation, \eqref{eq:HyNS} lose one derivative because of term $v_p\pa_y u_p$.   Paicu, Zhang and the Zhang \cite{PZZ} obtained the global well-posedness of system \eqref{eq:HyNS} when the initial data is small in the analytical space. Meantime, Renardy \cite{Renardy} also proved that the linearization of the hydrostatic Navier-Stokes equations at certain parallel shear flows is ill-posed, and may have a growth $e^{|k|t}$ which is the same as \eqref{eq:HyE} when the initial data is not convex. Thus, in order to obtain well-posedness  results that break through the analytic space, one may need the convexity condition on the velocity. For that, under the convexity condition, G\'erard-Varet, Masmoudi and Vicol 
proved the \eqref{eq:HyNS} is local well-posedness in the Gevrey class with index $9/8$ in \cite{GMV}. In \cite{GMV}, they firstly derive the vorticity equations $\om=\pa_y u$:
\beno
\pa_t (\pa_x \om)+  \pa_x v\pa_y \om+\cdots=0,
\eeno
where the worst term is  $\pa_x v$ leading to one derivative loss. 
Then, they use the "hydrostatic trick" which means that they take inner product with $\pa_x \om/\pa_y \om$ ($\pa_y\om\geq c_0>0$) instead of $\pa_x \om$ to take advantage of the cancellation:
\beno
\int \pa_x v\pa_y \om\cdot \f{\pa_x \om}{\pa_y \om}=\int \pa_x v\pa_x \om=-\int \pa_x\pa_y v\pa_x u=0.
\eeno
Such an idea was used previously in \cite{ MW1}. To close the energy estimates, "hydrostatic trick" is not enough due to the "bad" boundary condition of $\om$ 
\beno
 \pa_y\om|_{y=0}=-\pa_x \int_0^1 u^2dy+\cdots.
\eeno
which lose one derivative too. To overcome that, \cite{GMV} introduce the following decomposition
\beno
\om=\om^{bl}+\om^{in},
\eeno
where $\om^{bl}$ is the  boundary corrector which satisfies that
\beno
\pa_t \om^{bl}-\pa_{y}^2\om^{bl}=0,\quad \pa_y\om^{bl}|_{y=0}=-\pa_x \int_0^1 u^2dy.
\eeno
Following the above decomposition, \cite{GMV} obtain the well-posedness results of \eqref{eq:HyNS} in the Gevrey class with index $\gamma=\f98$. 

To search the best functional space for the system \eqref{eq:HyNS}, 
based on the Tollmien-Schlichting instabilities for Navier-Stokes \cite{GGN}, G\'erard-Varet, Masmoudi and Vicol  also give the following conjecture: {\it
"Our conjecture - based on a formal parallel with Tollmien-Schlichting instabilities for Navier-Stokes [18] - is that the best exponent possible should be $\gamma=\f32$, but such result is for the time being out of reach."} 

\medskip

During studying the anisotropic Navier-Stokes system \eqref{eq:ANS} and the hydrostatic Navier-Stokes/Prandtl system \eqref{eq:HyNS}, another important problem is to justify the inviscid limit. Under analytical setting, Paicu, Zhang and Zhang \cite{PZZ} justified the limit from \eqref{eq:ANS}  to  \eqref{eq:HyNS}. 
 Based on the work \cite{GMV}, we \cite{WWZ1} justified the limit in the Gevrey class with index $\gamma=\f98$. 

\medskip

In this paper, we aim to prove the conjecture of G\'erard-Varet, Masmoudi and Vicol. To do that, we use some idea from the classical inviscid limit theory. Next, we recall the recent development on the classical Prandtl equation and the inviscid limit theory. 
\medskip

There are a lot of papers studying the well-posedness of  Prandtl equation in some special functional space.  For monotonic initial data, \cite{Olei, AW, MW} used different method to get local existence and uniqueness of classical solutions to the Prandtl equation in Sobolev space. Without monotonic condition, \cite{LCS, SC1} proved that Prandtl equation are well-posedness in the analytic class; \cite{GM, LY, CWZ} proved  well-posedness of the Prandtl equations in Gevrey class for a class of concave initial data. Without any structure assumption, Dietert and G\'erard-Varet \cite{DG} proved well-posedness in Gevrey space with index $\gamma=2$. According to \cite{GD},  $\gamma=2$ may be the optimal index for the well-posedness theory.  For more results on the Prandtl equation, see \cite{GN, Wang, WXZ, XX, XZ}.

On the inviscid limit problem, we refer to \cite{SC2, WWZ, KVW, NN, Mae, FTZ} for the analytical class. Note that to go from analytic to Gevrey data is a challenging problem. The first results in  Gevrey  class is given by \cite{GMM}. G\'erard-Varet, Masmoudi and Maekawa \cite{GMM} proved stability of the Prandtl expansion for the perturbations in the Gevrey class when $U^{BL}(t, Y)$ is a monotone and concave function where the boundary layer is the shear type like
\beno
u_s^{\nu}=(U^e(t, y), 0)+(U^{BL}(t, \f{y}{\sqrt\nu}),0),
\eeno
where $\nu$ is the the viscosity coefficient.  Later, Chen, Wu and Zhang \cite{CWZ-1} improved the results in \cite{GMM} to get the $L^2\cap L^\infty$ stability. Very recently, G\'erard-Varet, Masmoudi and Maekawa \cite{GMM1} used a very 
clever decomposition to get the optimal Prandtl expansion around concave boundary layer.  Their results generalized the one obtained in \cite{GMM, CWZ-1} which restricted to expansions of shear flow type. In their paper, they decompose the  stream function $\phi$ as follows:
\beno
\phi=\phi_{slip}+\phi_{bc}
\eeno
  where $\phi_{slip}$ enjoys  a "good" boundary condition and $\phi_{bc}$ is a corrector which recover the boundary condition. This kind of decomposition is also used in \cite{CWZ-1}. To estimate $\phi_{bc}$, they also need the following decomposition
  \beno
  \phi_{bc}=\phi_{bc, S}+\phi_{bc, T}+\phi_{bc, R},
  \eeno
where $\phi_{bc, S}$ satisfies the Stokes equation, $\phi_{bc, T}$ is to correct the stretching term with "good" boundary condition and $\phi_{bc, R}$ solves formally the same system as $\phi_{slip}$.  In this paper, we apply the decomposition in \cite{GMM1} to justify the limit from \eqref{eq:ANS}  to \eqref{eq:HyNS}.  

\medskip
\subsection{Statement of the main results.}
Before stating the main results,  we give some assumption on initial data. Assume that initial data belong into the following Gevrey class:
\begin{align}\label{initial-1}
\|e^{\D^\f23}\pa_y u_0\|_{H^{14,0}}+\|e^{\D^\f23}\pa_y^3 u_0\|_{H^{10,0}}:=M<+\infty,
\end{align}
where $H^{r,s}$ is the  anisotropic Sobolev space defined by
\begin{align*}
\|f\|_{H^{r,0}}=\|\|f\|_{H^r_x(\mathbb{T})}\|_{H^s_y(0,1)}.
\end{align*}

%
%Let us assume some bound on  $(u_p^0, v_p^0, p_p^0)$ and $(u_p^2,v_p^2,p_p^2),$ where $(u_p^0, v_p^0, p_p^0):=(u_p, v_p, p_p)$ is the solution to system \eqref{eq:HyNS} and $(u_p^2,v_p^2,p_p^2)$ satisfies \eqref{HyNS-O(e^2)} with zero initial data. We assume under assumption \eqref{assume: convex}, system \eqref{eq:HyNS} and system \eqref{HyNS-O(e^2)} are well-posedness in Gevrey class with index $\gamma=\f32$ for $t\in[0,T],$ where $T\leq \min\{T_1, T_2\}$ and it holds that
%\begin{align}
%\|\pa_y u_p\|_{X^{13}}+\|\pa_y^3u_p\|_{X^9}\leq C,\label{bound: u_p}\\
%\|\pa_y u_p^2\|_{X^{10}}+\|\pa_y^3u_p^2\|_{X^6}\leq C,
%\label{bound: u_p^2}\\
%\pa_y^2u_p\geq 2c_0>0. \label{bound: convex}
%\end{align}
%
%Under  the bound \eqref{bound: u_p}-\eqref{bound: convex}, we state our main result in this paper.

\bigskip
More precisely, we consider initial data of the form
\begin{align*}
u^\e(0,x,y)=u_0(x,y),\quad v^\e(0,x,y)=v_0(x,y),
\end{align*}
which satisfy the compatibility conditions
\begin{align}
&\pa_xu_0+\pa_y v_0(t,x,y)=0, \quad u_0(t,x,0)=u_0(t,x,1)=v_0(t,x,0)=v_0(t,x,1)=0,\label{initial-2}\\
&\int_0^1\pa_x u_0dy=0,\quad \pa_y^2u_0|_{y=0,1}=\int_0^1(-\pa_xu_0^2+\pa_y^2 u_0)-\int_{\mathcal{S}}\pa_y^2 u_0 dxdy.\label{initial-3}
\end{align}
 Moreover, we assume the initial velocity satisfies the convex condition
\begin{align}\label{assume: convex}
\inf_{\mathcal{S}}\pa_y^2 u_0\geq 2 c_0>0.
\end{align}

Now, we are in the position to state the main results in our paper.
\begin{theorem}\label{thm: main}
Let initial data $u_0$ satisfy \eqref{initial-1}-\eqref{assume: convex}. Then there exist $T>0$ and $C > 0$ independent of $\e$ such that there exists a unique solution of the scaled anisotropic Navier–Stokes equations \eqref{eq:ANS} in $[0,T]$, which satisfies that for any $t\in[0,T],$ it holds that
\begin{align*}
\|(u^\e-u_p, \e v^\e-\e v_p)\|_{L^2_{x,y}\cap L^\infty_{x,y}}\leq C\e^2,
\end{align*}
where $(u_p, v_p)$ is the solution to \eqref{eq:HyNS}.
\end{theorem}

\begin{remark}\label{rmk}
Although we do not give the proof that the system \eqref{eq:HyNS} is well-posedness in Gevrey class $\f32$, one can follow the proof of Theorem \ref{thm: main} to obtain the well-posedness. To avoid repeatability in the proof, we omit the details. 
 Actually, the main difference between $\e=0$ and $\e\neq 0$ is on the construction boundary corrector $\phi_{bc,S}$, and readers can find more details in Remark \ref{rmk: difference}. 
\end{remark}

\begin{remark}
In the recent work \cite{GIM} by G\'erard-Varet, Iyer and  Maekawa, they establish  well-posedness of Hydrostatic Navier-Stokes system in Gevrey class $\f32$. In our present work, we focus on the inviscid limit problem.
\end{remark}

%\begin{remark}
%On the construction of boundary corrector, we start from the stream function $\phi$ instead of $\om$ in \cite{WWZ1}. The advantage is that  the normal velocity $v$ which is determined by $v=-\pa_x \phi$ has better regularity  than one given by Biot swart law $v=-\pa_x(-\tri_D)^{-1}\om.$ This makes us to improve the Gevrey index from $\f 98$ to $\f 32.$
%\end{remark}

\bigskip

\subsection{Sketch of the proof.}
In this subsection, we  sketch the  main ingredients in our proof.

\begin{enumerate}
\item {\bf Introduce the error equations.}
In Section 3, we deduce the error equations. We introduce the error
\begin{align*}
u^R=u^\e-u^p,\quad v^R=v^\e-v^p, \quad p^R=p^\e-p^p,
\end{align*}
which satisfies
\begin{align}\label{eq: error-1}
\left\{
\begin{aligned}
&\pa_t u^R-{\Delta_\e u^R}+v^R\pa_y u^p+\pa_x p^R=\cdots,\\
&\e^2(\pa_t v^R-\Delta_\e v^R)+\pa_y p^R=\cdots.\end{aligned}
\right.
\end{align}
Here $(u^p, v^p, p^p)$ is approximate solution given in \eqref{def: app}. The key point in this paper is to obtain the uniform estimate (in $\e$) of $(u^R, \e v^R)$ in the Gevrey class with index $\gamma=\f32.$ In view of \eqref{eq: error-1}, since $v^R$ is controlled via the relation
$v^R=-\int_0^y \pa_x u^R dy'$, the main difficulty comes from the term $v^R\pa_y u^p$, which loses one tangential derivative. In \cite{WWZ1}, we justify the limit in Gevrey class $\f 98.$ For the data in the Gevrey class with optimal index $\gamma=\f32$, we need to introduce new ideas.

\item {\bf Introduce the vorticity formulation.} In order to eliminate $p^R,$ we introduce vorticity $\om^R=-\e^2\pa_x v^R+\pa_y u^R$ and rewrite the equation of $\om^R$ by stream function $\phi$ which satisfies
\beno
v^R=-\pa_x\phi, \quad u^R=\pa_y \phi+C(t), \quad C(t)=\f1{2\pi}\int_{\cS} u^Rdxdy.
\eeno
Thus, we get
\begin{align}\label{eq: vorticity-1}
\left\{
\begin{aligned}
&(\pa_t-\Delta_\e)\Delta_\e \phi-\pa_x\phi\pa_y \om^p=\cdots,\qquad (x,y)\in\mathcal{S},\\
&\phi|_{y=0,1}=0, \quad \pa_y\phi|_{y=0,1}=C(t),\qquad x\in\mathbb{T}.
\end{aligned}
\right.
\end{align}
We notice the term $\pa_x\phi\pa_y \om^p$ also lose one tangential derivative. But under the convexity condition $\pa_y\om^p\geq c_0>0,$ one can use "hydrostatic trick" to deal with this term. Testing $\f{\om^R}{\pa_y\om^p}$ to the \eqref{eq: vorticity-1} instead of $\om^R,$ we have the following cancellation:
\begin{align*}
-\int_{\cS}\pa_x\phi \pa_y\om^p \cdot\f{\om^R}{\pa_y\om^p}dxdy=-\int_{\cS}\pa_x\phi \tri_\e \phi dxdy=\int_{\cS}\pa_x|\na_\e\phi|^2  dxdy=0,
\end{align*}
where we use $\phi|_{y=0,1}=0.$ However, the boundary condition of $\phi$ is $\pa_y\phi|_{y=0,1}= C(t)$ which brings an essential difficulty.

By the energy estimates, taking inner product in $X^{r}$(the definition is given in section 2) with $-\pa_t\phi,$ we get
\begin{align}\label{est: 1-intro}
&\sup_{s\in[0,t]}(\la\|\na_\e\phi(s)\|_{X^{\f73}}^2+\|\Delta_\e\phi(s)\|_{X^2}^2)\\
\nonumber
\leq&C\int_0^t(\e^{-2}\|\varphi\Delta_\e\phi\|_{X^2}^2+\e^{-2}\|\na_\e\phi\|_{X^2}^2+\cdots)ds,
\end{align}
where $\varphi(y)=y(1-y).$ All we need to do is to control $\e^{-1}\|\varphi\Delta_\e\phi\|_{X^2}$ and $\e^{-1}\|\na_\e\phi\|_{X^2}$ by the left hand side of \eqref{est: 1-intro}.

%Above estimate is terrible due to bad factor $\e^{-1}$ appears. In order to get estimate of velocity
%\begin{align*}
% \|(u^R, \e v^R)\|_{X^2}\leq C\|\om^R\|_{X^2}=C\|\tri_\e\phi\|_{X^2},
% \end{align*}
%  we should give estimate for $\e^{-1}\|\varphi\tri_\e\phi\|_{X^2}$ and $\e^{-1}\|\na_\e\phi\|_{X^2}$, which wish to control by the left hand side of \eqref{est: 1-intro}.

Motivated by \cite{GMM1},
we expect to achieve that by a decomposition of stream function $\phi=\phi_{slip}+\phi_{bc}$ in Gevrey $\f32$ regularity. Here $\phi_{slip}$ enjoys  a "good" boundary condition and $\phi_{bc}$ is a corrector which recover the boundary condition. In the following, we present the decomposition precisely.

\item {\bf Gevrey estimate under artificial boundary conditions.} 
$\phi_{slip}$ enjoying a good boundary condition is defined by
\begin{align}\label{eq: vorticity-good}
\left\{
\begin{aligned}
&(\pa_t-\Delta_\e)\om_{slip}-\pa_x\phi_{slip}\pa_y \om^p=\cdots,\qquad (x,y)\in\mathcal{S}\\
&\phi_{slip}|_{y=0,1}=0, \quad \om_{slip}|_{y=0,1}=0,\qquad x\in\mathbb{T}, 
\end{aligned}
\right.
\end{align}
where $\om_{slip}=\Delta_\e\phi_{slip}.$
By "hydrostatic trick" and Navier-slip boundary conditions to obtain
\begin{align}\label{est: 2-intro}
\la\int_0^t  (\|\om_{slip}\|_{X^{\f73}}^2+\|\na_\e\phi_{slip}\|_{X^\f73}^2+|\na_\e\phi_{slip}|_{y=0,1}|_{X^\f73}^2)ds\leq \f{C}{\la}\int_0^t\| \e\Delta_\e\phi \|_{X^{2}}^2 ds+\cdots.
\end{align}
 The full study of the Orr-Sommerfeld formulation \eqref{eq: vorticity-good} with Navier-slip boundary conditions is given in Section 7.

\item {\bf Recovery the non-slip boundary condition.} In Step (3), we use the slip boundary condition, not the real boundary condition $\pa_y\phi|_{y=0, 1}=C(t).$ To recover the boundary condition, we introduce the following system: 
\begin{align}\label{eq: vorticity-bc}
\left\{
\begin{aligned}
&(\pa_t-\Delta_\e)\phi_{bc}-\pa_x\phi_{bc}\pa_y \om^p=0, \qquad (x,y)\in\mathcal{S}\\
&\phi_{bc}|_{y=0,1}=0, \quad \pa_y\phi_{bc}|_{y=i}=h^i,\qquad x\in\mathbb{T},
\end{aligned}
\right.
\end{align}
where $\om_{bc}=\Delta_\e \phi_{bc}$ and $i=0,1.$ And we need to choose a suitable $h^i$ such that
\beno
\pa_y \phi_{bc}|_{y=0,1}=-\pa_y\phi_{slip}|_{y=0,1}+C(t).
\eeno
Next, we give the main idea for proving the existence of $h^i$:

\begin{itemize}

1. We define $\phi_{bc,S}=\phi_{bc,S}^0+\phi_{bc,S}^1$, where $\phi_{bc,S}^i$ solve
\begin{align}\label{eq: vorticity-bc11}
\left\{
\begin{aligned}
&(\pa_t -\Delta_\e)\Delta_\e\phi^i_{bc, S}=0,\\
&\phi^i_{bc, S}|_{y=0}=0,\quad \pa_y\phi^i_{bc, S}|_{y=i}=h^i,\\
&\phi^i_{bc, S}|_{t=0}=0,
\end{aligned}
\right.
\end{align}
where $x\in\mathbb{T},$ $y\in(0,+\infty)$ for $i=0$ and $y\in(-\infty,1)$ for $i=1.$ Taking Fourier transformation on $t$ and $x$, we can write the precise expression of the solution to obtain the Gevrey estimate for $\phi_{bc,S}^i$:
\begin{align}\label{est: 3-intro}
&\int_0^t\|\na_\e\phi^i_{bc,S}\|_{X^\f{5}{2}_i}^2+\|\varphi^i\Delta_\e\phi^i_{bc,S}\|_{X^\f52_i}^2 + \|\pa_x\phi^i_{bc,S}\|_{X^{\f53}}^2ds\leq \f{C}{\la^\f12}\int_0^t|h^i|_{X^\f73}^2ds,
\end{align}
where $\varphi^0(y)=y,~\varphi^1(y)=1-y.$ 
Compared with the decomposition in \cite{WWZ1}, we get more regularity of $\pa_x\phi^i_{bc,S}$ which is a key point to get the optimal Gevrey regularity. The details for this step is given in Section 8.1.

2. We correct the nonlocal term constructed in the above step, by consider the following equations: 
\begin{align}\label{eq: vorticity-bc12}
\left\{
\begin{aligned}
&(\pa_t -\Delta_\e)\Delta_\e\phi^i_{bc, R}-\pa_x\phi_{bc,R}^i\pa_y\om^p=\pa_x\phi_{bc,S}^i\pa_y\om^p,\quad (x,y)\in \cS\\
%&\phi^0_{bc, R}|_{y=0}=0,\quad \tri_\e\phi^i_{bc, R}|_{y=1}=-\phi_{bc,S}^0|_{y=1},\\
&\phi^i_{bc, R}|_{t=0}=0, \quad (x,y)\in \cS
\end{aligned}
\right.
\end{align}
with Navier-slip conditions.
% $\phi^0_{bc, R}|_{y=0}=0,\quad \tri_\e\phi^0_{bc, R}|_{y=1}=-\phi_{bc,S}^0|_{y=1}$ for $i=0$ and $\phi^1_{bc, R}|_{y=0}=-\phi_{bc,S}^1|_{y=0},\quad \tri_\e\phi^1_{bc, R}|_{y=1}=0$ for $i=1.$ 
By the same process as Step (3) and combining with the sharp estimate \eqref{est: 3-intro} to get estimate for $\phi^i_{bc,R}:$
\begin{align}\label{est: 4-intro}
\la\int_0^t  \|\om_{bc,R}^i\|_{X^{\f73}}^2ds&+\int_0^t (\|\na_\e\phi_{bc,R}^i\|_{X^\f73}^2+|\pa_y\phi_{bc,R}^i|_{y=0,1}|_{X^\f73}^2)ds\\
\nonumber
\leq&\f{C}{\la^\f12}\int_0^t|h^i|_{X^{\f73}}^2ds+\cdots,
\qquad t\in[0, T],
\end{align}

More details are given in Section 8.3.

3. We define $\phi_{bc}=\phi_{bc,S}+\phi_{bc,R}$ where
$\phi_{bc,S}=\sum_{i=0, 1}\phi^i_{bc,S} $ and $\phi_{bc,R}=\sum_{i=0, 1}\phi^i_{bc,R}$, which solve system \eqref{eq: vorticity-bc}.
To match the boundary condition on the  derivative of $\pa_y\phi|_{y=0,1}=C(t)$, we need
\beno
\pa_y\phi_{bc,S}|_{y=0,1}+\pa_y\phi_{bc,R}|_{y=0,1}=\pa_y\phi_{bc}|_{y=0,1}=- \pa_y\phi_{slip}|_{y=0,1}+C(t).
\eeno
On one hand,  $\phi_{bc,S}$ and $\phi_{bc,R}$ are defined by $h^i$. We define a $0$-order  operator $R_{bc}$ given in \eqref{def: R_bc} such that
\beno
(1+R_{bc}) h^i= - \pa_y\phi_{slip}|_{y=0,1}+C(t).
\eeno
Moreover, by the estimate in Step 1 and Step 2,  we can get
\begin{align*}
\int_0^t\Big|R_{bc}[h^0, h^1]\Big|_{X^{\f73}}^2 ds\leq&\f{C}{\la^\f12}\int_0^t|(h^0,h^1)|_{X^{\f 73}}^2ds.
\end{align*}
which means that $(1+R_{bc})$ is an invertible operator when $\lambda$ is large. That means that $\phi_{bc,S}$ and $\phi_{bc,R}$ are well-defined and \eqref{eq: vorticity-bc} is well-posedness. Details are given in Section 8.4.

Due to the transport terms,  we need to introduce a new auxiliary function $\phi_{bc, T}$ between Step 1 and Step 2.  For more details, see Section 8.2.

\end{itemize}

\item {\bf{Close the energy estimates \eqref{est: 1-intro}.}}
Summing estimates  \eqref{est: 3-intro} and \eqref{est: 4-intro} in Step (4), we get estimate for $\phi_{bc}:$
\begin{align*}
\int_0^t\|\na_\e\phi_{bc}\|_{X^\f{7}{3}}^2+\|\varphi\Delta_\e\phi_{bc}\|_{X^\f 73}^2ds\leq& \f{C}{\la^\f12}\int_0^t|(h^0, h^1)|_{X^\f73}^2ds\\
\leq&C\int_0^t |\na_\e\phi_{slip}|_{y=0,1}|^2_{X^\f73}ds+\cdots,
\end{align*}
which along with \eqref{est: 2-intro} to have
\begin{align*}
\int_0^t  (\|\varphi\om\|_{X^{\f73}}^2+\|\na_\e\phi\|_{X^\f73}^2)ds\leq \f{C}{\la}\int_0^t\| \e\Delta_\e\phi \|_{X^{2}}^2 ds+\cdots,
\end{align*}
then putting above estimate into \eqref{est: 1-intro} to close the estimate for system \eqref{eq: vorticity-1}.

\end{enumerate}

\subsection{Notations}

 \no -  $ \mathcal{S}^\e=\left\{(x, y) \in \mathbb{T}\times\mathbb{R} : 0<y<\varepsilon\right\}$ and $ \mathcal{S}=\left\{(x, y) \in \mathbb{T}\times\mathbb{R} : 0<y<1\right\}.$
 
 \no - $\na_\e=(\e\pa_x, \pa_y)$ and $\Delta_\e=\e^2 \pa_x^2+\pa_y^2$.
 
\no - Vorticity of Prandtl part $\om^p$ is defined by $ \om^p=\pa_y^2 u^p$.

\no - Vorticity of reminder part $\om^R=\Delta_\e \phi$ is defined by $\om^R=\e^2\pa_x v^R-\pa_y u^R$. In this paper, we also define $\om^i_{bc, j}=\Delta_\e \phi_{bc, j}^i$ where $i=0, 1$ and $j\in\{R, T\}$

\no - Cut-off functions $\varphi(y)=y(1-y)$ and $\varphi^i(y)=i+(-1)^iy$.

\no - $C(t)=\f1{2\pi}\int_{\cS} u^Rdxdy$.
 
\no - The Fourier transform of $f_{\Phi}$ is defined by $ e^{(1-\lambda t)\langle k\rangle^{\f23}} \widehat{f}(k)$.

\section{Gevrey norms and preliminary lemmas}

At the beginning of this section, we give the definition of the functional space $X^r$ and the Gevrey class.  First, we define 
\begin{align}\label{def: tau}
f_\Phi=\mathcal{F}^{-1}(e^{\Phi(t,k)}\widehat{f}(k))=e^{\Phi(t,D_x)}f,\quad \Phi(t,k)\eqdef\tau(t)\langle k\rangle^{\f23},
\end{align}
where $\tau(t)\ge 0$. Moreover, it is easy to get that $\Phi(t,k)$ satisfies the subadditive inequality
\begin{align}\label{eq:subadditive}
\Phi(t,k)\leq \Phi(t,k-\ell)+\Phi(t,\ell).
\end{align}

Now, we are in the position to define $X^r_{\tau}$ which is defined by
\begin{align*}
\|f\|_{X^r_{\tau}}=\|f_\Phi\|_{H^{r,0}}.
 \end{align*}
% where $\|\cdot\|_{H^{r,0}}$ is given by
% \begin{align*}
% \|f\|_{H^{r,0}}=\|\|f\|_{H^r_x(\mathbb{T})}\|_{L^2_y(0,1)}.
% \end{align*}
We say that  a function $f$ belongs to Gevrey class $\f32$ if $\|f\|_{X^r_{\tau}}<+\infty$.

 Moreover, we need to deal with some Gevrey class functions defined on the boundary. Thus, we introduce the following functional space:
\beno
|f|_{X^r_{\tau}}=\|f_\Phi\|_{H^r_x(\mathbb{T})}.
\eeno
where $f$ depends on variable $x$.

By the definition of $X^r_{\tau}$, it is easy to see that if $r'\geq r,$ then $\|\cdot\|_{X^{r'}_{\tau}}\geq \|\cdot\|_{X^{r}_{\tau}}.$  For simplicity, we drop subscript $\tau$ in the notations $\|f\|_{X^r_{\tau}}, |f|_{X^r_{\tau}}$ etc. In the sequel, we always take 
\beno
\tau(t)= 1-\la t ,  
\eeno
with $\la \geq 1$ determined later. Thus, if we take $t$ small enough, we have $\tau>0$.

\medskip

In the following, we present some lemmas on product estimates in Gevrey class and the readers can refer to Lemma 2.1-2.3 in \cite{WWZ} for details. The first lemma is the commutator estimate in Sobolev space:
\begin{lemma}\label{lem:com-S}
Let $r\ge 0,~s_1>\f32$, $s>\f12$ and $0\leq\delta\leq 1$. Then it holds that
\begin{align*}
&\big\|[\langle D\rangle^r,f]\pa_xg\big\|_{L^2_x}\leq C\|f\|_{H^{s_1}_x}\|g\|_{H^r_x}+C\|f\|_{H^{r+1-\d}_x}\|g\|_{H^{s+\d}_x}.
\end{align*}

\end{lemma}

In Gevrey class, we have
\begin{lemma}\label{lem:product-Gev}
Let $r\ge 0$ and $s> \f12$. Then  it holds that
\beno
|fg|_{X^r}\leq C  |f|_{X^{s}}   |g |_{X^r}  +  C |f |_{X^r}  |g |_{X^{s}}.
 \eeno
\end{lemma}

For the commutator in Gevrey class, we have
\begin{lemma}\label{lem:com-Gev}
Let $r\ge 0,~s_1>\f32$, $s>\f12$ and $0\leq\delta\leq 1$. Then it holds that
\begin{align*}
&\|(f\pa_xg)_{\Phi}-f\pa_xg_{\Phi}\|_{H^r_x}\leq C|f|_{X^{s_1}}|g|_{X^{r+\f23}}+C|f|_{X^{r+1-\delta}}|g|_{X^{s+\delta}}.
\end{align*}

\end{lemma}

\section{Approximate equations and Error equations}

\subsection{Approximate equations}
By Hilbert asymptotic method, we can obtain the  approximate solutions. We define approximate solutions as following
\begin{align}\label{def: app}
\left\{
\begin{aligned}
u^p(t,x,y)=&u_p^0(t,x,y)+\e^2u_p^2(t,x,y),\\
v^p(t,x,y)=&v_p^0(t,x,y)+\e^2v_p^2(t,x,y),\\
p^p(t,x,y)=&p_p^0(t,x,y)+\e^2p_p^2(t,x,y),
\end{aligned}
\right.
\end{align}
 where $(u_p^0, v_p^0, p_p^0)$ satisfies equation \eqref{eq:HyNS} and $(u_p^2, v_p^2, p_p^2)$ satisfies equation
\begin{align}\label{HyNS-O(e^2)}
\left\{
\begin{aligned}
&\pa_t u_p^2+ u_p^0\pa_x u_p^2+v_p^0\pa_y u_p^2+u_p^2\pa_x u_p^0+v_p^2\pa_y u_p^0+\pa_x p_p^2-\pa^2_y u_p^2=-\pa_x^2 u_p^0,\\
&\pa_y p_p^2=-(\pa_t v_p^0+u_p^0\pa_x v_p^0+v_p^0\pa_y v_p^0-\pa_y^2v_p^0),\\
&\pa_x u_p^2+\pa_y v_p^2=0,\\
&(u_p^2, v_p^2)|_{y=0,1}=0,\\
&u_p^2|_{t=0}=0.
\end{aligned}
\right.
\end{align}
Based on the equation of  $(u_p^0, v_p^0, p_p^0)$ and $(u_p^2, v_p^2, p_p^2)$, we deduce approximate solution $(u^p, v^p, p^p)$ which satisfies the following equation:
\begin{align}
\left\{
\begin{aligned}
&\pa_t u^p+u^p\pa_x u^p+v^p\pa_y u^p+\pa_x p^p-\Delta_\e u^p=-R_1,\\
&\e^2(\pa_t v^p+u^p\pa_x v^p+v^p\pa_y v^p-\Delta_\e v^p)+\pa_y p^p=-R_2,\\
&\pa_x u^p+\pa_y v^p=0,\\
&(u^p, v^p)|_{y=0,1}=0,\\
&(u^p, v^p)|_{t=0}=(u_0, v_0),
\end{aligned}
\right.
\end{align}
where reminder $(R_1, R_2)$ is given by
\begin{align}
R_1=&\e^4(u_p^2\pa_x u_p^2+v_p^2\pa_y v_p^2-\pa_x^2 u_p^2),\label{R_1}\\
R_2=&\e^4\Big(\pa_t v_p^2+u_p^0\pa_x v_p^2+u_p^2\pa_x v_p^0+\e^2u_p^2\pa_x v_p^2+v_p^0\pa_y v_p^2+v_p^2\pa_y v_p^0  \label{R_2}\\
\nonumber
&\quad+\e^2 v_p^2\pa_y v_p^2-\pa_x^2(v_p^0+\e^2 v_p^2)-\pa_y ^2 v_p^2\Big).
\end{align}

By the definition of $R_1$ and $R_2$, it is easy to get that
\begin{align*}
(R_1, R_2)\sim O(\e^4).
\end{align*}
 
\medskip

\subsection{Equations of error functions  }
We define error functions $(u^R, v^R, p^R)$:
\begin{align*}
u^R=u^\e-u^p,\quad v^R=v^\e-v^p,\quad p^R=p^\e-p^p.
\end{align*}
It is easy to deduce the system of error functions:
\begin{align}\label{eq:error-u}
\left\{
\begin{aligned}
&\pa_t u^R-\Delta_\e u^R+\pa_x p^R+u^\e\pa_x u^R+u^R\pa_x u^p+v^\e\pa_y u^R+v^R\pa_y u^p=R_1,\\
&\e^2(\pa_t v^R-\Delta_\e v^R+u^\e\pa_x v^R+u^R\pa_x v^p+v^\e\pa_y v^R+v^R\pa_y v^p)+\pa_y p^R=R_2,\\
&\pa_x u^R+\pa_y v^R=0,\\
&(u^R,v^R)|_{y=0}=(u^R,v^R)|_{y=1}=0,\\ 
&(u^R,v^R)|_{t=0}=0.
\end{aligned}
\right.
\end{align}
For convenience, we rewrite \eqref{eq:error-u} as
\begin{align}\label{eq: error-(u,v)-1}
\left\{
\begin{aligned}
&\pa_t u^R-{\Delta_\e u^R}+u^p\pa_x u^R+u^R\pa_x u^p+v^R\pa_y u^p+v^p\pa_y u^R+\pa_x p^R=\mathcal{N}_u+R_1,\\
&\e^2(\pa_t v^R-\Delta_\e v^R+u^p\pa_x v^R+u^R\pa_x v^p+v^R\pa_y v^p+v^p\pa_y v^R)+\pa_y p^R=\e^2\mathcal{N}_v+R_2,\\
&\pa_x u^R+\pa_y v^R=0 ,\\
&(u^R,v^R)|_{y=0}=(u^R,v^R)|_{y=1}=0,\\
&(u^R,v^R)|_{t=0}=0.
\end{aligned}
\right.
\end{align}
Here $(\mathcal{N}_u,\mathcal{N}_v)$ is nonlinear term  given  by
\begin{align}\label{def: (N_u, N_v)}
\mathcal{N}_u=&-(u^R\pa_x u^R+v^R\pa_y u^R),\quad \mathcal{N}_v=-(u^R\pa_x v^R+v^R\pa_y v^R).
\end{align}

\medskip

Based on the above system, we get the equations of the vorticity $\om^R=\pa_y u^R-\e^2 \pa_x v^R$:
 \begin{align}\label{eq: om^R}
&\pa_t \om^R-\Delta_\e\om^R+u^p\pa_x \om^R+u^R\pa_x \om^p+v^p\pa_y \om^R+v^R\pa_y \om^p \\
&\quad=\pa_y \mathcal{N}_u-\e^2\pa_x \mathcal{N}_v+\e^2 f_1+f_2,\nonumber
 \end{align}
 where  $f_1, f_2$ are defined by
 \begin{align}
&f_1=-(u^R\pa_x^2 v^p+v^R\pa_x\pa_y v^p),\label{eq: f_1}\\
&f_2= \pa_y R_1-\e^2\pa_x R_2,\label{eq: f_2}\\
&\om^p=\pa_y u^p.\label{eq: om-R}
\end{align}

% \begin{align}
% \mathcal{N}=\pa_y \mathcal{N}_u-\e^2\pa_x \mathcal{N}_v=-u^R\pa_x \om^R  -v^R\pa_y \om^R.\label{eq:N1}
% \end{align}

Moreover, following the calculations in \cite{WWZ}, we can obtain the boundary conditions of $\om^R$:
\begin{align}
&(\pa_y+\e |D|)\om^R|_{y=0}=\pa_y(\Delta_{\e,D})^{-1}(f-\mathcal{N})|_{y=0}+\f1 {2\pi}\int_{\cS}\pa_tu^Rdxdy,\label{BC: om-0}\\
&(\pa_y-\e |D|)\om^R|_{y=1}=\pa_y(\Delta_{\e,D})^{-1}(f-\mathcal{N})|_{y=1}+\f 1 {2\pi}\int_{\cS}\pa_tu^Rdxdy,\label{BC: om-1}
\end{align}
where
\begin{align}
 \mathcal{N}=&\pa_y \mathcal{N}_u-\e^2\pa_x \mathcal{N}_v=-u^R\pa_x \om^R  -v^R\pa_y \om^R.\label{eq:N1}\\
f=&f_3-\e^2 f_1-f_2,\label{eq: f}\\
f_3=& u^p\pa_x \om^R+u^R\pa_x \om^p+v^p\pa_y \om^R+v^R\pa_y \om^p,\quad \om^p=\pa_y u^p.\label{eq: f_3}
\end{align}

%
%\begin{align*}
%\int_{\cS} \pa_t u^R dxdy=\int_{\cS}\pa_y^2 u^R dxdy=\int_{\cS} \pa_y\om^R dxdy,
%\end{align*}
%which gives
%\begin{align}\label{eq:ut-est}
%\Big|\int_{\cS} \pa_tu^R dxdy\Big|\leq& \|\pa_y\om^R\|_{L^1}.
%\end{align}
% 

\medskip

\subsection{Equations of stream function}
Thanks to $\pa_x u^R+\pa_yv^R=0$ and $v^R|_{y=0,1}=0$, there exists a stream function $\phi$ satisfying the following system:  
\begin{align}\label{def: (u^R, v^R)}
-\pa_x\phi=v^R,\quad \pa_y\phi=u^R-\f1{2\pi}\int_{\cS} u^Rdxdy,
\end{align}
Since $\int_{\mathbb{T}} v^Rdx=0,$ the function $\phi$ is periodic in $x$. Thanks to $\pa_x\phi|_{y=0,1}=0$ and $\phi(1,x)-\phi(0,x)=0$, we may assume that $\phi|_{y=0,1}=0$. Thus, there holds that 
\begin{align}\label{eq: tri_e phi=om^R}
\Delta_\e \phi=\om^R\quad {\rm{in}}\,\,\cS,\quad \phi|_{y={0,1}}=0.
\end{align}
Taking \eqref{def: (u^R, v^R)} and \eqref{eq: tri_e phi=om^R} into \eqref{eq: om^R} and using the boundary condition $(u^R, v^R)|_{y=0,1}=0$, 
 we obtain
\begin{align}\label{eq: phi}
\left\{
\begin{aligned}
&(\pa_t -\Delta_\e)\Delta_\e\phi+u^p\pa_x \Delta_\e\phi+v^p\pa_y \Delta_\e\phi+\pa_y\phi\pa_x \om^p-\pa_x\phi\pa_y \om^p\\
&\qquad\qquad\qquad\qquad\qquad=\pa_y \mathcal{N}_u-\e^2\pa_x \mathcal{N}_v +\e^2 f_1+f_2-C(t)\pa_x\om^p,\\
&\phi|_{y=0,1}=0,\quad \pa_y\phi|_{y=0,1}=C(t),
\end{aligned}
\right.
\end{align}
where $C(t)=\f1{2\pi}\int_{\cS} u^Rdxdy$ and $(\mathcal{N}_u, \mathcal{N}_v), ~f_1, ~f_2$ are given in \eqref{def: (N_u, N_v)}, \eqref{eq: f_1} and  \eqref{eq: f_2}.

\medskip

In the end of subsection, we state some elliptic estimates which can be got by classical theory. First, by elliptic estimate and Hardy inequality, we have
\begin{align}\label{est: na_e phi}
\|\na_\e\phi\|_{L^2}\leq C\|\varphi\om^R\|_{L^2},
\end{align}
where $\varphi(y)=y(1-y)$ and $\na_\e=(\pa_y,\e\pa_x).$

Since $(u^R,v^R)$ satisfies the following elliptic equations
\beno
  \left\{
   \begin{array}{ll}
    \Delta_\e u^R=\pa_y\om^R,\\
    u^R|_{y=0,1}=0,
   \end{array}
  \right.
  \qquad
    \left\{
     \begin{array}{ll}
      \Delta_\e v^R=-\pa_x \om^R,\\
      v^R|_{y=0,1}=0,
     \end{array}
    \right.
\eeno
we arrive at
\begin{align}\label{eq:uR-elliptic}
\big\|(u^R,\e v^R, \pa_y u^R, \e \pa_x u^R, \e\pa_y v^R,\e^2\pa_x v^R)\big\|_{X^r}\leq C\|\om^R\|_{X^r},
\end{align}
for any $r\geq0.$

%
%The goal in this paper is to give uniformly (in $\e$) for system \eqref{eq: phi} in Gevrey regularity with index $\gamma=\f32$ in $x$ direction and Sobolev regularity in $y$ direction.

\bigskip

\section{Estimate of $\na_\e\phi$ and $\Delta_\e\phi$ in Gevrey space}

Before giving the estimate of $\na_\e\phi$ and $\Delta_\e\phi$, we need the estimates of the reminder terms $R_1$ and $R_2$ which are defined by the approximate solution $u^p$ and $v^p$. 
For $(u^p, v^p),$ we have the following bound:
\begin{lemma}\label{lem: u^p}
Let initial data $u_0$ of \eqref{eq:HyNS} satisfy \eqref{initial-1}-\eqref{assume: convex}.
There exists a time $T_p$ such that $(u_p^i, v_p^i),~i=0,2$ defined in \eqref{eq:HyNS} and \eqref{HyNS-O(e^2)} have the following estimates 
\begin{align*}
&\|v_p^0\|_{X^{11}}+\|(u_p^0, \e v_p^0)\|_{X^{12}}+\|\pa_y u_p^0\|_{X^{12}}+\|\pa_y^3 u_p^0\|_{X^8}\leq C,
\\
&\|v_p^2\|_{X^9}+\|(u_p^2, \e v_p^2)\|_{X^{10}}+\|\pa_y u_p^2\|_{X^{10}}+\|\pa_y^3 u_p^2\|_{X^6}\leq C,
\end{align*}
for $t\in [0, T_p]$.

Moreover, according to \eqref{def: app}, it holds that
\begin{align*}
\|v^p\|_{X^9}+\|(u^p, \e v^p)\|_{X^{10}}+\|\pa_y u^p\|_{X^{10}}+\|\pa_y^3 u^p\|_{X^6}\leq C,\qquad t\in [0, T_p].
\end{align*}
and
\beno
\pa_{y} \om^p\geq  c_0 ,\qquad t\in [0, T_{p}].
\eeno
 
\end{lemma}
\begin{proof}
Here,  the key of this lemma is to prove the \eqref{eq:HyNS} is well-posedness in the Gevrey class $\f32$ which is the conjecture in \cite{GMV}. If we set $\e=0$ and follow the step by step  in this paper, we can get the  the conjecture proved. Here, to avoid the repeatability, we leave the proof to the readers.

\end{proof}

\medskip

Then, by the definition of $(R_1, R_2)$ in \eqref{R_1}-\eqref{R_2}, using Lemma \ref{lem:product-Gev} to get that
\begin{lemma}\label{lem: R}
It holds that
\begin{align*}
\|(R_1, R_2)\|_{X^3}\leq C\e^4,\quad \|\na(R_1, R_2)\|_{X^2}\leq C\e^4,\quad t\in [0, T_p].
\end{align*}
\end{lemma}

\medskip

Now, we state our main result in this section:
\begin{proposition}\label{pro: (na_e phi, tri_e phi)}
There exist $0<T<\min\{T_p, \f{1}{2\lambda}\}$ and $\la_0\geq1,$ such that for any $t\in[0,T]$ and $\la\geq \la_0,$ it holds that
\begin{align*}
&\sup_{s\in[0,t]}\Big(\la\|\na_\e\phi(s)\|_{X^{\f73}}^2+\|\Delta_\e\phi(s)\|_{X^2}^2\Big) +\int_0^t \Big(\|\pa_t \na_{\e}\phi_\Phi\|^2_{H^{2,0}}+\|\na_{\e}\om^R\|^2_{L^2}\Big)ds\\
\leq&C\int_0^t\Big(\e^{-2}\|\varphi\tri_\e\phi\|_{X^2}^2+\e^{-2}\|\na_\e\phi\|_{X^2}^2+\|\Delta_\e\phi\|_{X^2}^2+\|(\mathcal{N}_u, \e\mathcal{N}_v)\|_{X^2}^2+\|\mathcal{N}\|_{L^2}^2+\e^8\Big)ds,
\end{align*}
where $\Delta_\e\phi=\om^R$, $\varphi(y)=y(1-y)$ and  $C$ is a  constant  independent of $\e$.
\end{proposition}
\begin{proof}

%
%We complete this proposition by combining Lemma \ref{lem: (na_e phi, tri_e phi)} and Lemma \ref{lem: ||om^R||_L^2} together.
%\end{proof}
%\medskip
%
%\begin{lemma}\label{lem: (na_e phi, tri_e phi)}
%There exist $T>0$ and $\la_0\geq1,$ such that for any $t\in[0,T]$ and $\la\geq \la_0,$ it holds that
%\begin{align*}
%&\sup_{s\in[0,t]}(\la\|\na_\e\phi(t)\|_{X^{\f73}}^2+\|\tri_\e\phi(t)\|_{X^2}^2)+\int_0^t \|\pa_t \na_{\e}\phi_\Phi\|^2_{H^{2,0}}\\
%\leq&C\int_0^t(\e^{-2}\|\varphi\tri_\e\phi\|_{X^2}^2+\e^{-2}\|\na_\e\phi\|_{X^2}^2+\|(\mathcal{N}_u, \e\mathcal{N}_v)\|_{X^2}^2+\|\tri_\e\phi\|_{X^2}^2+\|\pa_y\om^R\|_{L^2}^2+\e^8)ds,
%\end{align*}
%where $C$ is a  constant independent of $\e$.
%\end{lemma}
%\begin{proof}
%Recall the equations \eqref{eq: phi}. 

Acting $e^{\Phi(t, D_x)}$ on the both sides of the first equation of \eqref{eq: phi}, we get
\begin{align*}
(\pa_t+\la\D^\f23 &-\Delta_\e)\Delta_\e\phi_\Phi+(u^p\pa_x \Delta_\e\phi+v^p\pa_y \Delta_\e\phi)_\Phi+(\pa_y\phi\pa_x \om^p-\pa_x\phi\pa_y \om^p)_\Phi\\
&=\pa_y (\mathcal{N}_u)_\Phi-\e^2\pa_x (\mathcal{N}_v)_\Phi+(\e^2 f_1+f_2)_\Phi
-C(t)\pa_x \om^p_\Phi
\end{align*}
Taking $H^{2,0}$ inner product with $-\pa_t\phi_\Phi$ and using boundary conditions
\begin{align*}
\phi_\Phi|_{y=0,1}=0,\quad \pa_y\phi_\Phi|_{y=0,1}=C(t),
\end{align*}
 we integrate by parts to arrive at
 \begin{align}\label{est: (na_e phi, tri_e phi)}
 &\f12\f{d}{dt}(\la\|\na_\e\phi\|_{X^{\f73}}^2+\|\Delta_\e\phi\|_{X^2}^2)+\|\pa_t\na_\e\phi_\Phi\|_{H^{2,0}}^2-\Big\langle \Delta_\e\phi_\Phi,\pa_t\pa_y\phi_\Phi\Big\rangle_{H^{2}_x}\Big|_{y=0}^{y=1}\\
 \nonumber
 =&\Big\langle (u^p\pa_x \Delta_\e\phi+v^p\pa_y \Delta_\e\phi)_\Phi,\pa_t\phi_\Phi \Big\rangle_{H^{2,0}}+\Big\langle (\pa_y\phi\pa_x \om^p-\pa_x\phi\pa_y \om^p)_\Phi,\pa_t\phi_\Phi\Big\rangle_{H^{2,0}}\\
 \nonumber
 &+\Big\langle \pa_y (\mathcal{N}_u)_\Phi-\e^2\pa_x (\mathcal{N}_v)_\Phi,-\pa_t\phi_\Phi \Big\rangle_{H^{2,0}}+\Big\langle (\e^2 f_1+f_2)_\Phi,-\pa_t\phi_\Phi \Big\rangle_{H^{2,0}}\\
&-\Big\langle C(t)\pa_x \om^p_\Phi,-\pa_t\phi_\Phi \Big\rangle_{H^{2,0}} \nonumber
 \\
 \nonumber
  =&I_1+\cdots+I_5.
 \end{align}

Firstly, let's estimate of $I_i,~i=1,\cdots, 5$ term by term.

\underline{Estimate of $I_1.$} Since divergence free condition $\pa_x u^p+\pa_y v^p=0,$ we get
\begin{align*}
(u^p\pa_x \Delta_\e\phi+v^p\pa_y \Delta_\e\phi)_\Phi=\pa_x(u^p \Delta_\e\phi)_\Phi+\pa_y(v^p \Delta_\e\phi)_\Phi.
\end{align*}
According to $(u^p, v^p)|_{y=0,1}=0,$ we use integration by parts and Lemma \ref{lem:product-Gev} to have
\begin{align*}
I_1=&-\Big\langle (u^p\Delta_\e\phi)_\Phi,\pa_t\pa_x\phi_\Phi \Big\rangle_{H^{2,0}}-\Big\langle (v^p\Delta_\e\phi)_\Phi,\pa_t\pa_y\phi_\Phi \Big\rangle_{H^{2,0}}\\
\leq&C\Big\||\f{u^p}{\varphi}|_{X^2}|\varphi\Delta_\e\phi|_{X^2}\Big\|_{L^2_y}\|\pa_t\pa_x\phi_\Phi\|_{H^{2,0}}+C\Big\||\f{v^p}{\varphi}|_{X^2}|\varphi\Delta_\e\phi|_{X^2}\Big\|_{L^2_y}\|\pa_t\pa_y\phi_\Phi\|_{H^{2,0}}\\
\leq&C\||\f{u^p}{\varphi}|_{X^2}\|_{L^\infty_y}\|\varphi\Delta_\e\phi\|_{X^2}\|\pa_t\pa_x\phi_\Phi\|_{H^{2,0}}+C\Big\||\f{v^p}{\varphi}|_{X^2}\|_{L^\infty_y}\|\varphi\Delta_\e\phi\|_{X^2}\|\pa_t\pa_y\phi_\Phi\|_{H^{2,0}}\\
\leq&C\e^{-1}\|\varphi\Delta_\e\phi\|_{X^2}\|\pa_t\na_\e\phi_\Phi\|_{H^{2,0}}.
\end{align*}

\underline{Estimate of $I_2.$} Similarly, we write
\begin{align*}
(\pa_y\phi\pa_x \om^p-\pa_x\phi\pa_y \om^p)_\Phi=\Big(\pa_x(\pa_y\phi\om^p)-\pa_y(\pa_x\phi \om^p)\Big)_\Phi,
\end{align*}
then along with $\phi|_{y=0,1}=0$, we  use integration by parts and Lemma \ref{lem:product-Gev} to deduce
\begin{align*}
I_2=&\Big\langle\pa_x(\pa_y\phi\om^p)_\Phi-\pa_y(\pa_x\phi_\Phi \om^p)_\Phi,\pa_t\phi_\Phi\Big\rangle_{H^{2,0}}\\
=&-\Big\langle(\pa_y\phi\om^p)_\Phi,\pa_t\pa_x\phi_\Phi\Big\rangle_{H^{2,0}}+\Big\langle(\pa_x\phi \om^p)_\Phi,\pa_t\pa_y\phi_\Phi\Big\rangle_{H^{2,0}}\\
\leq&C\||\om^p|_{X^2}\|_{L^\infty_y}\|\pa_y \phi\|_{X^2}\|\pa_t\pa_x\phi_\Phi\|_{H^{2,0}}+C\||\om^p|_{X^2}\|_{L^\infty_y}\|\pa_x \phi\|_{X^2}\|\pa_t\pa_y\phi_\Phi\|_{H^{2,0}}\\
\leq&C\e^{-1}\|\na_\e\phi\|_{X^2}\|\pa_t\na_\e\phi_\Phi\|_{H^{2,0}}.
\end{align*}

\underline{Estimate of $I_3.$} Due to $\phi|_{y=0,1}=0,$ taking integration by parts, it yields that
\begin{align*}
I_3\leq C\|(\mathcal{N}_u, \e\mathcal{N}_v)\|_{X^2}\|\pa_t\na_\e\phi_\Phi\|_{H^{2,0}}.
\end{align*}

\underline{Estimate of $I_4.$} Recall $f_1$ and $f_2$ in \eqref{eq: f_1}-\eqref{eq: f_2}. According to \eqref{eq:uR-elliptic} and Lemma \ref{lem: R}, we have
\begin{align*}
I_4\leq& C(\|\e^2 f_1\|_{X^2}+\|f_2\|_{X^2})\|\pa_t\phi_\Phi\|_{H^{2,0}}\\
\leq& C(\e^2\|u^R\|_{X^2}+\e\|\e v^R\|_{X^2}+\e^4)\|\pa_t\pa_y\phi_\Phi\|_{H^{2,0}}\\
\leq& C(\e\|\Delta_\e\phi\|_{X^2}+\e^4)\|\pa_t\pa_y\phi_\Phi\|_{H^{2,0}}.
\end{align*}

\underline{Estimate of $I_5.$} Poincar\'e inequality implies
\begin{align*}
\|\pa_t\phi_\Phi\|_{H^{2,0}}\leq C\|\pa_t\pa_y\phi_\Phi\|_{H^{2,0}},
\end{align*}
for  $\phi|_{y=0,1}=0.$ Since 
\beno
|C(t)|=|\f 1 {2\pi}\int_{\cS} u^Rdxdy|\leq C\|u^R\|_{L^2}\leq C\|\om^R\|_{L^2}\leq C
\|\Delta_\e\phi\|_{L^2},
\eeno
 we get
\begin{align*}
I_5\leq&C\|\Delta_\e\phi\|_{L^2}\|\pa_t\pa_y\phi_\Phi\|_{H^{2,0}}.
\end{align*}

Collecting $I_1-I_5$ together, it holds that
\begin{align}\label{est: (na_e phi, tri_e phi)-I_1-I_5}
\nonumber
I_1+\cdots+I_5\leq&C\e^{-1}\|\varphi\Delta_\e\phi\|_{X^2}\|\pa_t\na_\e\phi_\Phi\|_{H^{2,0}}+C\e^{-1}\|\na_\e\phi\|_{X^2}\|\pa_t\na_\e\phi_\Phi\|_{H^{2,0}}\\
&+C\|(\mathcal{N}_u, \e\mathcal{N}_v)\|_{X^2}\|\pa_t\na_\e\phi_\Phi\|_{H^{2,0}}+C\|\Delta_\e\phi\|_{L^2}\|\pa_t\pa_y\phi_\Phi\|_{H^{2,0}}\\
\nonumber
&+C(\e\|\Delta_\e\phi\|_{X^2}+\e^4)\|\pa_t\pa_y\phi_\Phi\|_{H^{2,0}}\\
\nonumber
\leq&\f1{10}\|\pa_t\na_\e\phi_\Phi\|_{H^{2,0}}^2+C\Big(\e^{-2}\|\varphi\Delta_\e\phi\|_{X^2}^2+\e^{-2}\|\na_\e\phi\|_{X^2}^2+\|\Delta_\e\phi\|_{X^2}^2\Big)\\
\nonumber
&+C\Big(\|(\mathcal{N}_u, \e\mathcal{N}_v)\|_{X^2}+\e^8\Big).
\end{align}

Next, we focus on the boundary term $\Big\langle \Delta_\e\phi_\Phi,\pa_t\pa_y\phi_\Phi\Big\rangle_{H^{2}_x}\Big|_{y=0}^{y=1}$. 
First, we give the estimate of $C'(t)$. 
\begin{align*}
\int_{\cS} \pa_t u^R dxdy=\int_{\cS}\pa_y^2 u^R dxdy=\int_{\cS} \pa_y\om^R dxdy,
\end{align*}
which gives
\begin{align}\label{eq:ut-est}
\Big|\int_{\cS} \pa_tu^R dxdy\Big|\leq& \|\pa_y\om^R\|_{L^1}.
\end{align}

Owing to 
$$\pa_t\pa_y\widehat{\phi_\Phi}|_{y=0,1}(k)=C'(t)\d(k),$$
where $\d(k)$ is a dirac function and $k\in\mathbb{Z},$ we have
\begin{align*}
\Big\langle \Delta_\e\phi_\Phi,\pa_t\pa_y\phi_\Phi\Big\rangle_{H^2_x}\Big|_{y=0}^{y=1}=&\Big\langle \Delta_\e\phi\Big|_{y=0}^{y=1},C'(t)\Big\rangle_{L^2_x}=\Big\langle \int_0^1\pa_y\Delta_\e\phi dy,C'(t)\Big\rangle_{L^2_x}\\
\leq& CC'(t)\|\pa_y\om^R\|_{L^2}\leq C\|\pa_y\om^R\|_{L^2}^2,
\end{align*}
where we used \eqref{eq:ut-est} in the last step.

Putting above estimate and \eqref{est: (na_e phi, tri_e phi)-I_1-I_5} into \eqref{est: (na_e phi, tri_e phi)}, we get
\begin{align}\label{est: 1}
\f12\f{d}{dt}(\la\|\na_\e\phi\|_{X^{\f73}}^2&+\|\Delta_\e\phi\|_{X^2}^2)\nonumber\\
\leq&C\Big(\e^{-2}\|\varphi\Delta_\e\phi\|_{X^2}^2+\e^{-2}\|\na_\e\phi\|_{X^2}^2+\|(\mathcal{N}_u, \e\mathcal{N}_v)\|_{X^2}^2+\|\Delta_\e\phi\|_{X^2}^2+\|\pa_y\om^R\|^2_{L^2}+\e^8\Big) 
\end{align}
%We obtain the resulting estimate by integrating time from $0$ to $t.$

\medskip
Next, we give the estimates of $\|\pa_y\om^R\|^2_{L^2}$. Firstly, recalling the equation of $\om^R$:
\begin{align}\label{eq: om^R-1}
&\pa_t \om^R-\Delta_\e\om^R+u^p\pa_x \om^R+u^R\pa_x \om^p+v^p\pa_y \om^R+v^R\pa_y \om^p \\
&=\pa_y \mathcal{N}_u-\e^2\pa_x \mathcal{N}_v+\e^2 f_1+f_2,\nonumber
 \end{align}
 with boundary conditions
 \begin{align}
&(\pa_y+\e |D|)\om^R|_{y=0}=\pa_y(\Delta_{\e,D})^{-1}(f-\mathcal{N})|_{y=0}+\f1 {2\pi}\int_{\cS}\pa_tu^Rdxdy,\label{BC: om^R-0}\\
&(\pa_y-\e |D|)\om^R|_{y=1}=\pa_y(\Delta_{\e,D})^{-1}(f-\mathcal{N})|_{y=1}+\f 1 {2\pi}\int_{\cS}\pa_tu^Rdxdy,\label{BC: om^R-1}
\end{align}
where $f_1,~f_2, f$ and $\mathcal{N}$ are given in \eqref{eq: f_1}-\eqref{eq: f}.

Taking $L^2$ inner product with $\om^R$ on \eqref{eq: om^R-1} and integration by parts, it follows from $(\mathcal{N}_u,\e\mathcal{N}_v)|_{y=0,1}=0$ and $(u^p, v^p)|_{y=0,1}=0$ to obtain
\begin{align}\label{est: ||om^R||_L^2}
\f12\f{d}{dt}\|\om^R\|_{L^2}^2+&\|\na_\e\om^R\|_{L^2}^2-\int_{\mathbb{T}}\pa_y\om^R \om^Rdx\Big|_{y=0}^{y=1}\\
\nonumber
\leq&C\|(u^R, v^R)\|_{L^2}\|\om^R\|_{L^2}+C\|(\mathcal{N}_u,\e\mathcal{N}_v)\|_{L^2}\|\na_\e\om^R\|_{L^2}\\
\nonumber
&+C(\|\e^2 u^R\|_{L^2}+\|\e^2v^R\|_{L^2}+\e^4)\|\om^R\|_{L^2}\\
\nonumber
\leq&\f1{10}\|\na_\e\om^R\|_{L^2}^2+C(\|(\mathcal{N}_u,\e\mathcal{N}_v)\|_{L^2}^2+\|\om^R\|_{H^{1,0}}^2+\e^8).
\end{align}
For the boundary term, we use \eqref{BC: om^R-0}-\eqref{BC: om^R-1} to write
\begin{align*}
\int_{\mathbb{T}}\pa_y\om^R \om^Rdx\Big|_{y=0}^{y=1}=&\int_{\mathbb{T}}\Big(\e|D|\om^R|_{y=1}+\pa_y(\Delta_{\e,D})^{-1}(f-\mathcal{N})|_{y=1}+C(t)\Big)\om^R|_{y=1} dx\\
&-\int_{\mathbb{T}}\Big(-\e|D|\om^R|_{y=0}+\pa_y(\Delta_{\e,D})^{-1}(f-\mathcal{N})|_{y=0}+C(t)\Big)\om^R|_{y=0} dx\\
=&\int_{\mathbb{T}}(\e|D|\om^R~\om^R)|_{y=0,1}dx+C(t)\int_{\mathbb{T}}\om^R|_{y=0}^{y=1}dx\\
&+\int_{\mathcal{S}}\pa_y\Big(\pa_y(\Delta_{\e,D})^{-1}(f-\mathcal{N})\om^R\Big)dxdy=B_1+B_2+B_3.
\end{align*}

Let $y_0\in [0,1]$ so that 
\beno
\|\e|D|{\om^R}(y_0)\|_{L^2_x}\le \|\e|D|\om^{R}\|_{L^2},
\eeno
then along with Gagliardo-Nirenberg inequality
 \begin{align}\label{equality: GN}
\|g\|_{L^\infty_y}\leq C\|g\|_{L^2_y}^\f12\big(\|g\|_{L^2_y}^\f12+\|\pa_yg\|_{L^2_y}^\f12\big),
\end{align}
 it infers that
\begin{align*}
B_1=&\int_{y_0}^1\pa_y(\e|D|\om^R\om^R)dxdy+\int_{y_0}^0\pa_y(\e|D|\om^R\om^R)dxdy+2\int_{\mathbb{T}}(\e|D|\om^R~\om^R)|_{y=y_0}dx\\
\leq&C\|\e|D|\om^R\|_{L^2}\|\pa_y\om^R\|_{L^2}+C\|\e|D|\om^{R}\|_{L^2}\|\om^R\|_{L^\infty_y(L^2_x)}\\
\leq&C\e\|\om^R\|_{H^{1,0}}^2+C\e\|\pa_y\om^R\|_{L^2}^2
\end{align*}

Similarly, we use \eqref{equality: GN} and $|C(t)|\leq C\|u^R\|_{L^2}\leq C\|\om^R\|_{L^2}$ to have
\begin{align*}
B_2\leq&C|C(t)|\|\om^R\|_{L^\infty_y(L^2_x)}\leq C\|\om^R\|_{L^2}^\f32(\|\om^R\|_{L^2}^\f12+\|\pa_y\om^R\|_{L^2}^\f12)\\
\leq&\f1{10}\|\pa_y\om^R\|_{L^2}^2+C\|\om^R\|_{L^2}^2.
\end{align*}

All we left is to do $B_3.$ With the fact: operator $\pa_y(\tri_{\e,D})^{-1},~\pa_y(\tri_{\e,D})^{-1}(\pa_y,\e\pa_x)$ and $\pa_y^2(\tri_{\e,D})^{-1}$ are bounded from $L^2\to L^2,$  we have
\begin{align*}
B_3=&\int_{\mathcal{S}}\pa_y^2(\tri_{\e,D})^{-1}(f-\mathcal{N})\om^Rdxdy+\int_{\mathcal{S}}\pa_y(\tri_{\e,D})^{-1}\pa_y(v^p\om^R)\pa_y\om^Rdxdy\\
&+\int_{\mathcal{S}}\pa_y(\tri_{\e,D})^{-1}(f-\pa_y(v^p\om^R)-\pa_y \mathcal{N}_u-\e^2\pa_x \mathcal{N}_v)\pa_y\om^Rdxdy\\
\leq&C\|f-\mathcal{N}\|_{L^2}\|\om^R\|_{L^2}+C\|v^p\om^R\|_{L^2}\|\pa_y\om^R\|_{L^2}\\
&+C(\|f-\pa_y(v^p\om^R)\|_{L^2}+\|(\mathcal{N}_u,\e\mathcal{N}_v)\|_{L^2})\|\pa_y\om^R\|_{L^2}.
\end{align*}

According to the definition of \eqref{eq: f} and \eqref{eq:N1}, we have
\begin{align*}
\|f\|_{L^2}\leq &C(\|\pa_x\om^R\|_{L^2}+\|\pa_y\om^R\|_{L^2}+\|(u^R, v^R)\|_{L^2}+\e^4)\\
\leq&C(\|\om^R\|_{H^{1,0}}+\|\pa_y\om^R\|_{L^2}+\e^4),
\end{align*}
and
\begin{align*}
\|f-\pa_y(v^p\om^R)\|_{L^2}\leq& C(\|\pa_x\om^R\|_{L^2}+\|(u^R, v^R)\|_{L^2}+\e^4)\\
\leq&C(\|\om^R\|_{H^{1,0}}+\e^4),
\end{align*}
which give that
\begin{align*}
B_3\leq& C(\|\om^R\|_{H^{1,0}}+\|\pa_y\om^R\|_{L^2}+\|\mathcal{N}\|_{L^2}+\e^4)\|\om^R\|_{L^2}\\
&+C(\|\om^R\|_{H^{1,0}}+\|(\mathcal{N}_u,\e\mathcal{N}_v)\|_{L^2}+\e^4)\|\pa_y\om^R\|_{L^2}\\
\leq&\f1{10}\|\pa_y\om^R\|_{L^2}^2+C(\|\om^R\|_{H^{1,0}}^2+\|(\mathcal{N}_u,\e\mathcal{N}_v)\|_{L^2}^2+\|\mathcal{N}\|_{L^2}^2+\e^8).
\end{align*}

Summarizing $B_1-B_3$ together, we obtain
\begin{align}\label{est: ||om^R||_L^2-BC}
\Big|\int_{\mathbb{T}}\pa_y\om^R \om^Rdx\Big|_{y=0}^{y=1}\Big|\leq&(\f1{5}+C\e)\|\pa_y\om^R\|_{L^2}^2+C(\|\om^R\|_{H^{1,0}}^2+\|(\mathcal{N}_u,\e\mathcal{N}_v)\|_{L^2}^2+\|\mathcal{N}\|_{L^2}^2+\e^8).
\end{align}
Substituting \eqref{est: ||om^R||_L^2-BC} into \eqref{est: ||om^R||_L^2}, we take $\e$ small enough to arrive at
\begin{align*}
\f12\f{d}{dt}\|\om^R\|_{L^2}^2+&\f12\|\na_\e\om^R\|_{L^2}^2\\
\nonumber
\leq&C(\|(\mathcal{N}_u,\e\mathcal{N}_v)\|_{L^2}^2+\|\om^R\|_{H^{1,0}}^2+\e^8)\\
\nonumber
&+C(\|\om^R\|_{H^{1,0}}^2+\|(\mathcal{N}_u,\e\mathcal{N}_v)\|_{L^2}^2+\|\mathcal{N}\|_{L^2}^2+\e^8)\\
\nonumber
\leq&C(\|\Delta_{\e}\phi\|_{H^{1,0}}^2+\|(\mathcal{N}_u,\e\mathcal{N}_v)\|_{L^2}^2+\|\mathcal{N}\|_{L^2}^2+\e^8).
\end{align*}

Bring the above estimate into \eqref{est: 1} and  integrate time from $0$ to $t$ to get the desired results.

\end{proof}

\section{Sketch the proof to Theorem \ref{thm: main}}
In this section, we shall sketch the proof of Theorem \ref{thm: main}.  In the paper, we use the continue argument. Here, we define
\ben\label{assume: 1}
T^*\eqdef \sup\{t >0 |  \sup_{s\in[0,t]}\|\om^R\|_{X^2}\leq \mathfrak{C} \e^3\}.
\een

\subsection{The key  {\it a priori} estimates.} In this subsection, we shall present the key {\it a priori} estimates used in the proof of Theorem \ref{thm: main}.

By Proposition \ref{pro: (na_e phi, tri_e phi)}, we need the estimates of $\int_0^t (\|\na_\e\phi\|_{X^2}^2+\|\varphi\Delta_\e\phi\|_{X^2}^2)ds$ to close the energy. 
\begin{proposition}\label{pro: key}
Let $\phi$ be the solution of \eqref{eq: phi}. Then there exists $\la_0\geq1$ and $0<T<\min\{T_p, \f{1}{2\lambda}\}$ such that for $\la\geq \la_0$ and $t\in[0, T]$, it holds that
\begin{align}\label{est: linear-1}
\int_0^t (\|\na_\e\phi\|_{X^2}^2+\|\varphi\Delta_\e\phi\|_{X^2}^2)ds\leq&C\int_0^t(\|(\mathcal{N}_u,\e \mathcal{N}_v)\|_{X^2}^2+\|\e \Delta_\e \phi\|_{X^\f 53}^2+\e^8)ds,
\end{align}
 with $t\in[0, T]$.

\end{proposition}
 
The proof of the above proposition is the main part in this paper and we prove it in the section 6.

\bigskip

\subsection{Proof of Theorem \ref{thm: main}}

Before we prove the Theorem \ref{thm: main}, we firstly give the estimates for the nonlinear terms:
\begin{proposition}\label{prop: N}
Under the assumption \eqref{assume: 1}, there holds that
\begin{align}
\int_0^t\|(\mathcal{N}_u,\e \mathcal{N}_v)\|_{X^2}^2ds\leq&C\e^4\int_0^t\|\om^R\|_{X^2}^2ds,\label{est: nonlinear-1 }\\
\int_0^t\|\mathcal{N}\|_{L^2}^2ds\leq&C\e^4\int_0^t\|\na_\e \om^R\|_{L^2}^2ds, \label{est: nonlinear-2}
\end{align}
where $t\in [0, T^*]$.

\end{proposition}

\begin{proof}
By the definition of $\mathcal{N}_u,$ we have
\begin{align*}
\int_0^t\|\mathcal{N}_u\|_{X^2}^2ds\leq \int_0^t\|u^R\pa_x u^R\|_{X^2}^2ds+\int_0^t\|v^R\pa_y u^R\|_{X^2}^2ds=I_1+I_2.
\end{align*}
It follows from Lemma \ref{lem:product-Gev} and \eqref{eq:uR-elliptic} that
\begin{align*}
I_1\leq&C\int_0^t \|\f{u^R_\Phi}{\e}\|_{L^\infty_y(H^2_x)}^2\|\e\pa_x u^R\|_{X^2}^2ds\\
\leq& C\e^{-2}\int_0^t \|u^R\|_{X^2}(\|u^R\|_{X^2}+\|\pa_y u^R\|_{X^2})\|\e\pa_x u^R\|_{X^2}^2ds\\
\leq& C\e^{-2}\int_0^t \|\om^R\|_{X^2}^4ds,
\end{align*}
where we use Gagliardo-Nirenberg inequality \eqref{equality: GN} in the second step.

Similarly, we use Lemma \ref{lem:product-Gev} and \eqref{eq:uR-elliptic} to deduce
\begin{align*}
I_2\leq&C\int_0^t\|v^R_\Phi\|_{L^\infty_y(H^2_x)}^2\|\pa_y u^R\|_{X^2}^2ds 
\leq C\e^{-2}\int_0^t\|\e\pa_xu^R\|_{X^2}^2\|\pa_y u^R\|_{X^2}^2ds\\
\leq&C\e^{-2}\int_0^t\|\om^R\|_{X^2}^4ds,
\end{align*}
where we use $v^R=-\int_0^y \pa_x u^R dy'$ in the second step.

Collecting $I_1$ and $I_2$ together and using \eqref{assume: 1}, it holds that
\begin{align*}
\int_0^t\|\mathcal{N}_u\|_{X^2}^2ds\leq C\e^{-2}\int_0^t\|\om^R\|_{X^2}^4ds\leq C\e^4\int_0^t\|\om^R\|_{X^2}^2ds.
\end{align*}

The estimate for $\e \mathcal{N}_v$ is obtained by changing $u^R$ into $\e v^R$ in the above argument and we omit details. Thus we obtain \eqref{est: nonlinear-1 }.

\medskip

For \eqref{est: nonlinear-2}, we use the definition of $\mathcal{N}$ to have
\begin{align*}
\int_0^t\|\mathcal{N}\|_{L^2}^2ds\leq& \int_0^t\|u^R\pa_x \om^R\|_{L^2}^2ds+\int_0^t\| v^R\pa_y\om^R\|_{L^2}^2ds\\
\leq&\e^{-2}\int_0^t\|u^R\|_{L^\infty_y(H^1_x)}^2\|\e\pa_x \om^R\|_{L^2}^2ds+\int_0^t\| v^R\|_{L^\infty_y(H^1_x)}^2\|\pa_y\om^R\|_{L^2}^2ds\\
\leq&C\e^{-2}\int_0^t\|\om^R\|_{H^{1,0}}^2\|\e\pa_x \om^R\|_{L^2}^2ds+C\int_0^t\| \om^R\|_{H^{2,0}}^2\|\pa_y\om^R\|_{L^2}^2ds\\
\leq&C\e^{-2}\sup_{s\in[0,t]}\|\om^R\|_{X^2}^2\int_0^t\|\na_\e \om^R\|_{L^2}^2ds\\
\leq&C\e^4\int_0^t\|\na_\e \om^R\|_{L^2}^2ds,
\end{align*}
by \eqref{assume: 1} and we obtain \eqref{est: nonlinear-2}.

\end{proof}

\medskip

With Proposition \ref{pro: key} and Proposition \ref{prop: N} in hand, we are in the position to prove the Theorem \ref{thm: main}. By Proposition \ref{pro: (na_e phi, tri_e phi)}, Proposition \ref{pro: key} and Proposition \ref{prop: N}, we get 
\beno
\sup_{s\in[0,t]}(\la\|\na_\e\phi(s)\|_{X^{\f73}}^2+\|\Delta_\e\phi(s)\|_{X^2}^2) +\int_0^t \|\pa_t \na_{\e}\phi_\Phi\|^2_{H^{2,0}}\leq Ct \e^6+C\int_0^t \|\Delta_\e\phi(s)\|_{X^2}^2ds,
\eeno
for $t\in[0,T].$ By Gronwall inequality and choosing a small $T<\min\{T_p, \f{1}{2\lambda}\}$, we get that 
\beno
\sup_{s\in[0,t]}(\la\|\na_\e\phi(s)\|_{X^{\f73}}^2+\|\om^R\|_{X^2}^2) +\int_0^t \|\pa_t \na_{\e}\phi_\Phi\|^2_{H^{2,0}}\leq \f{\mathfrak{C}}{2} \e^6.
\eeno

By Sobolev embedding theorem and Lemma \ref{lem: u^p}, we get the Theorem \ref{thm: main} proved.

\bigskip

\section{The proof of Proposition \ref{pro: key}}
All we left is the Proposition \ref{pro: key}. To prove that,  we firstly give the following decomposition of $\phi$:
\begin{align}\label{decom: phi}
\phi=\phi_{slip}+\phi_{bc},
\end{align}
where $\phi_{slip}$ satisfies that
 \begin{align}\label{eq: tri_e phi-good}
\left\{
\begin{aligned}
&(\pa_t -\Delta_\e)\Delta_\e\phi_{slip}+u^p\pa_x \Delta_\e\phi_{slip}+v^p\pa_y \Delta_\e\phi_{slip}+\pa_y\phi_{slip}\pa_x \om^p-\pa_x\phi_{slip}\pa_y \om^p\\
&\qquad\qquad\qquad\qquad\qquad=\pa_y \mathcal{N}_u-\e^2\pa_x \mathcal{N}_v +\e^2 f_1+f_2-C(t)\pa_x\om^p,\\
&\phi_{slip}|_{y=0,1}=0,\quad \Delta_\e\phi_{slip}|_{y=0,1}=0,\\
&\phi_{slip}|_{t=0}=0,
\end{aligned}
\right.
\end{align}
and $\phi_{bc}$ satisfies that
\begin{align}
\left\{
\begin{aligned}
&(\pa_t-\Delta_\e)\Delta_\e\phi_{bc}+u^p\pa_x \Delta_\e\phi_{bc}+v^p\pa_y \Delta_\e\phi_{bc}+\pa_y\phi_{bc}\pa_x \om^p-\pa_x\phi_{bc}\pa_y \om^p=0,\\
&\phi_{bc}|_{y=0,1}=0,\quad \pa_y\phi_{bc}|_{y=0,1}=- \pa_y\phi_{slip}|_{y=0,1}+C(t)\\
&\phi_{bc}|_{t=0}=0.
\end{aligned}
\right.
\end{align}

\medskip

To prove Proposition \ref{pro: key}, we need the estimates of $\phi_{slip}$ and $\phi_{bc}$. First, we notice that $\phi_{slip}$ has a good boundary condition. We use "hydrostatic trick" method to get its estimates. The proof of following proposition is given in section 7. 
\begin{proposition}\label{pro: om-good}
There exists $\la_0>1$and $0<T<\min\{T_p, \f{1}{2\lambda}\}$ such that for all $t\in[0,T]$, $\la\geq \la_0$, there holds that 
\begin{align*}
\|\Delta_\e\phi_{slip}\|_{X^{2}}^2+&\la\int_0^t  (\|\Delta_\e\phi_{slip}\|_{X^{\f73}}^2+\|\na_\e\phi_{slip}\|_{X^\f73}^2+|\na_\e\phi_{slip}|_{y=0,1}|_{X^\f73}^2)ds+\int_0^t\|\na_\e\Delta_\e\phi_{slip}\|_{X^{2}}^2ds\\
\leq&C\int_0^t\|(\mathcal{N}_u,\e \mathcal{N}_v)\|_{X^2}^2ds+\f{C}{\la}\int_0^t\|(C(t),\e^2 f_1,f_2)\|_{X^{\f53}}^2 ds.
\end{align*}

\end{proposition}

\medskip

The estimates of $\phi_{bc}$ is much more difficult. Here, we state the main results on it:
\begin{proposition}\label{pro: phi_{bc}}
There exists $\la_0>1$and $0<T<\min\{T_p, \f{1}{2\lambda}\}$ such that for all $t\in[0,T]$, $\la\geq \la_0$, there holds that \begin{align}\label{est: na_e phi_app}
\int_0^t \|\na_\e\phi_{bc} \|_{X^\f73}^2+\|\varphi \Delta_\e \phi_{bc}\|^2_{X^2}ds\leq& \f{C}{\la^\f12}\int_0^t \Big(|\na_\e\phi_{slip}|_{y=0,1}|^2_{X^\f73}+|C(s)|^2\Big)ds,
\end{align}
where $C$ is a universal constant.
\end{proposition}
The proof of Proposition \ref{pro: phi_{bc}} is given in section 8. 

\medskip

Based on the above two propositions, we are in the position to prove Proposition \ref{pro: key}. 
Firstly, we give the estimates of $\|u^R\|_{L^2}$ which is used to control the $C(t)$.
\begin{lemma}\label{lem: ||(u^R, e v^R)||_L^2}
There exist $0<T<\min\{T_p, \f{1}{2\lambda}\}$ and $\la_0\geq1$  such that for $t\in[0,T]$ and $\la\geq \la_0,$ it holds that
\begin{align}\label{est: |(u^R, e v^R)|_L^2}
\|e^{(1-\la t)}(u^R,\e v^R)\|_{L^2}^2+&\la\int_0^t\|e^{(1-\la s)}(u^R,\e v^R)\|_{L^2}^2ds+\int_0^t\|e^{(1-\la s)}\na_\e(u^R,\e v^R)\|_{L^2}^2ds\\
\nonumber
\leq&C\int_0^t(\|e^{(1-\la s)}(\mathcal{N}_u,\e\mathcal{N}_v)\|_{L^2}^2+\e^8)ds+\f{C}{\la}\int_0^t\|\na_\e\phi\|_{X^2}^2ds.
\end{align}

\end{lemma}
\begin{remark}
We use weighted quantity $\|e^{(1-\la t)}(u^R,\e v^R)\|_{L^2}$ instead of $\|(u^R,\e v^R)\|_{L^2}$ to obtain small constant factor in front of $\int_0^t\|\na_\e\phi\|_{X^2}^2ds$ in \eqref{est: |(u^R, e v^R)|_L^2}.
\end{remark}

\begin{proof}
Taking $L^2$ inner product with $e^{2(1-\la t)}u^R$ in the first equation of \eqref{eq: error-(u,v)-1} and with $e^{2(1-\la t)}v^R$ in the second equation of \eqref{eq: error-(u,v)-1}, we use the fact $$\pa_t(e^{2(1-\la t)}f)=e^{2(1-\la t)}\pa_tf+2\la e^{2(1-\la t)}f$$ and integrate by parts by boundary condition $(u^R, v^R)|_{y=0,1}=0$ to yield
\begin{align*}
\f12\f{d}{dt}&\|e^{(1-\la t)}(u^R, \e v^R)\|_{L^2}^2+\la\|e^{(1-\la t)}(u^R, \e v^R)\|_{L^2}^2+\|e^{(1-\la t)}\na_\e(u^R,\e v^R)\|_{L^2}^2\\
\leq&C(\|e^{(1-\la t)}(u^R, v^R)\|_{L^2}+\|e^{(1-\la t)}(\mathcal{N}_u,\e\mathcal{N}_v)\|_{L^2}+\|(R_1, R_2)\|_{L^2})\|e^{(1-\la t)}(u^R, \e v^R)\|_{L^2}\\
\leq&\f{\la}{2}\|e^{(1-\la t)}(u^R, \e v^R)\|_{L^2}^2+C(\|e^{(1-\la t)}(\mathcal{N}_u,\e\mathcal{N}_v)\|_{L^2}^2+\e^8)+\f{C}{\la}\|e^{(1-\la t)}\pa_x u^R\|_{L^2}^2,
\end{align*}
where we write $v^R=-\int_0^y \pa_x u^Rdy'$ and use the fact $\pa_xu^p+\pa_y v^p=0 $ to eliminate transport term and $\pa_xu^R+\pa_y v^R=0 $ to eliminate pressure term respectively.

Afterwards, integrating time from $0$ to $t$ and using $\pa_x u^R=-\pa_x \pa_y\phi$, we obtain
\begin{align*}
\|e^{(1-\la t)}(u^R,\e v^R)(t)\|_{L^2}^2+&\la\int_0^t\|e^{(1-\la s)}(u^R,\e v^R)\|_{L^2}^2ds+\int_0^t\|e^{(1-\la s)}\na_\e(u^R,\e v^R)\|_{L^2}^2ds\\
\leq&C\int_0^t(\|e^{(1-\la s)}(\mathcal{N}_u,\e\mathcal{N}_v)\|_{L^2}^2+\e^8)ds+\f{C}{\la}\int_0^t\|e^{(1-\la s)}\pa_x \pa_y\phi\|_{L^2}^2ds.
\end{align*}
Finally, we use  $\|e^{(1-\la s)}\pa_x\pa_y\phi\|_{L^2}\leq C\|\na_\e\phi\|_{X^2}$ to complete the proof.

\end{proof}

\medskip

{\bf{Proof of Proposition \ref{pro: key}:}} Now, we give the proof Proposition \ref{pro: key}.  We divide this proof into two parts.

\underline{ Estimates of $\int_0^t \|\na_\e\phi\|_{X^2}^2$.} Since 
\beno
|C(t)|=|\f 1 {2\pi}\int_{\cS} u^Rdxdy|\leq C\|u^R\|_{L^2}\leq C\|e^{(1-\la t)}u^R\|_{L^2},
\eeno
by Lemma \ref{lem: ||(u^R, e v^R)||_L^2} to ensure
\begin{align}\label{est: |C(t)|-1}
|C(t)|^2\leq& C\int_0^t(\|e^{(1-\la s)}(\mathcal{N}_u,\e\mathcal{N}_v)\|_{L^2}^2+\e^8)ds+\f{C}{\la}\int_0^t\|\na_\e\phi\|_{X^2}^2ds.
\end{align}
By the definition of $f_1$ and $f_2$, we obtain that
\begin{align*}
&\int_0^t \|f_1\|_{X^\f 53}^2+\|f_2\|_{X^\f 53}^2ds\leq C\int_0^t (\|\e \tri_\e \phi\|_{X^\f 53}^2+\e^8)ds,
\end{align*}
 we get
 \begin{align}\label{est: phi_slip}
 \la\int_0^t  (\|\Delta_\e\phi_{slip}\|_{X^{\f73}}^2+&\|\na_\e\phi_{slip}\|_{X^\f73}^2+|\na_\e\phi_{slip}|_{y=0,1}|_{X^\f73}^2)ds\\
 \nonumber
 \leq&C\int_0^t\|(\mathcal{N}_u,\e \mathcal{N}_v)\|_{X^2}^2ds+\f{C}{\la}\Big(|C(t)|^2+\int_0^t (\|\e \Delta_\e \phi\|_{X^\f 53}^2+\e^8)ds\Big).
 \end{align}
Then, it follows $\phi=\phi_{slip}+\phi_{bc}$ and \eqref{est: na_e phi_app} to deduce
\begin{align}\label{est: ||na_e phi||_X^2-1}
\int_0^t \|\na_\e\phi\|_{X^2}^2ds\leq&\int_0^t \|\na_\e\phi_{slip}\|_{X^2}^2ds+\int_0^t \|\na_\e\phi_{bc}\|_{X^2}^2ds\\
\nonumber
\leq&\f{C}{\la}\int_0^t\|(\mathcal{N}_u,\e \mathcal{N}_v)\|_{X^2}^2ds+\f{C}{\la^2}\Big(\int_0^t \|\e \Delta_\e \phi\|_{X^\f 53}^2+\e^8ds\Big)\\
\nonumber
&+\f{C}{\la^\f12}\int_0^t |\na_\e\phi_{slip}|_{y=0,1}|_{X^\f73}^2 ds+\f{C}{\la^\f12}|C(t)|^2,\\
\leq&  \f{C}{\la^\f12}|C(t)|^2+\f{C}{\la}\int_0^t\|(\mathcal{N}_u,\e \mathcal{N}_v)\|_{X^2}^2ds+\f{C}{\la^2}\int_0^t (\|\e \Delta_\e \phi\|_{X^\f 53}^2+\e^8)ds .\nonumber
\end{align}
Plusing \eqref{est: |C(t)|-1} and  above estimates together and taking $\la$ large enough to get
\begin{align}\label{est: ||na_e phi||_X^2-2}
|C(t)|^2+ \int_0^t \|\na_\e\phi\|_{X^2}^2ds  \leq C\int_0^t\|(\mathcal{N}_u,\e \mathcal{N}_v)\|_{X^2}^2ds+C\int_0^t (\|\e \Delta_\e \phi\|_{X^\f 53}^2+\e^8)ds . 
\end{align}

\underline{ Estimates of $\int_0^t\|\varphi\tri_\e\phi \|_{X^2}^2$.} 

It follows from \eqref{est: phi_slip}  and  \eqref{est: ||na_e phi||_X^2-2} that
\begin{align*}
\int_0^t\|\Delta_\e\phi_{slip}\|_{X^2}^2ds\leq& \f{C}{\la}\int_0^t\|(\mathcal{N}_u,\e \mathcal{N}_v)\|_{X^2}^2ds+\f{C}{\la}\Big(|C(t)|^2+\int_0^t (\|\e \Delta_\e \phi\|_{X^\f 53}^2+\e^8)ds\Big)\\
\leq&\f{C}{\la}\int_0^t(\|(\mathcal{N}_u,\e \mathcal{N}_v)\|_{X^2}^2+\|\e \Delta_\e \phi\|_{X^\f 53}^2+\e^8)ds.
\end{align*}

Applying Proposition \ref{pro: phi_{bc}} again, we get
\beno
\int_0^t\|\varphi\Delta_\e\phi_{bc}\|_{X^2}^2ds\leq C\int_0^t(\|(\mathcal{N}_u,\e \mathcal{N}_v)\|_{X^2}^2+\|\e \Delta_\e \phi\|_{X^\f 53}^2+\e^8)ds.
\eeno

Combing above two estimates, we get
\beno
\int_0^t\|\varphi\Delta_\e\phi \|_{X^2}^2ds\leq C\int_0^t(\|(\mathcal{N}_u,\e \mathcal{N}_v)\|_{X^2}^2+\|\e \Delta_\e \phi\|_{X^\f 53}^2+\e^8)ds.
\eeno

By now, we get the desired results.

\bigskip

\section{Vorticity estimates under artificial boundary condition: Proof of Proposition \ref{pro: om-good}}

In the section, we give the proof of Proposition \ref{pro: om-good}. To simplify the notation, we drop the subscript in the system \eqref{eq: tri_e phi-good}:
\begin{align}\label{eq: tri_e phi-good-1}
\left\{
\begin{aligned}
&(\pa_t -\Delta_\e)\Delta_\e\phi+u^p\pa_x \Delta_\e\phi+v^p\pa_y \Delta_\e\phi+\pa_y\phi\pa_x \om^p-\pa_x\phi\pa_y \om^p\\
&\qquad \qquad\qquad=\pa_y \mathcal{N}_u-\e^2\pa_x \mathcal{N}_v +\e^2 f_1+f_2-C(t)\pa_x\om^p,\\
&\phi|_{y=0,1}=0,\quad \Delta_\e\phi|_{y=0,1}=0,\\
&\phi|_{t=0}=0.
\end{aligned}
\right.
\end{align}

The goal in this section is to establish uniform (in $\e$) estimate of vorticity $\om=\tri_\e\phi.$
\begin{proposition}\label{pro: om-good-1}
There exists $\la_0>0$ and $0<T<\min\{T_p, \f{1}{2\lambda}\}$ such that for all $t\in[0,T]$, $\la\geq \la_0$, the following holds that 
\begin{align*}
\|\om(t)\|_{X^{2}}^2+&\la\int_0^t  (\|\om\|_{X^{\f73}}^2+\|\na_\e\phi\|_{X^\f73}^2+|\na_\e\phi|_{y=0,1}|_{X^\f73}^2)ds+\int_0^t\|\na_\e\om\|_{X^{2}}^2ds\\
\leq&C\int_0^t\|(\mathcal{N}_u,\e \mathcal{N}_v)\|_{X^2}^2ds+\f{C}{\la}\int_0^t\|\e^2 f_1, f_2,C(t)\|_{X^{\f53}}^2 ds.
\end{align*}

\end{proposition}

\begin{proof}
By Lemma \ref{lem: u^p}, we have
\beno
\pa_y \om^p\geq c_0>0.
\eeno
Hence, we use the "hydrostatic trick" to get the desired results. Firstly, acting operator $e^{\Phi(t,D_x)}$ on the first equation of \eqref{eq: tri_e phi-good} to get
\begin{align}\label{eq: om_Phi-good}
&(\pa_t+\la\D^\f23 -\Delta_\e)\om_\Phi+u^p\pa_x \om_\Phi+v^p\pa_y\om_\Phi-\pa_x\phi_\Phi\pa_y \om^p\\
\nonumber
&\qquad=-(\pa_y\phi\pa_x \om^p)_\Phi-[e^{\Phi(t,D_x)},u^p\pa_x ] \om-[e^{\Phi(t,D_x)},v^p\pa_y ] \om \\
\nonumber
&\qquad\quad+[e^{\Phi(t,D_x)},\pa_y\om^p ]\pa_x \phi +\pa_y (\mathcal{N}_u)_\Phi-\e^2\pa_x (\mathcal{N}_v)_\Phi+(\e^2 f_1+f_2-C(t)\pa_x\om^p)_\Phi.
\end{align}
In view of \eqref{eq: om_Phi-good}, the terrible term comes from $\pa_x\phi_\Phi\pa_y \om^p,$ which lose one tangential derivative. In order to overcome the derivative loss,  we take $\D^2$ on the \eqref{eq: om_Phi-good} and then take $L^2$ inner product with $\f{\D^2\om_\Phi}{\pa_y\om^p}$ to obtain that
\begin{align*}
\f12\f{d}{dt}&\Big\|\f{\D^2\om_\Phi}{\sqrt{\pa_y\om^p}}\Big\|_{L^2}^2+\la\Big\|\f{\D^{\f73}\om_\Phi}{\sqrt{\pa_y\om^p}}\Big\|_{L^2}^2+\Big\|\f{\na_\e\D^2\om_\Phi}{\sqrt{\pa_y\om^p}}\Big\|_{L^2}^2\\
=&-\int_{\mathcal{S}}\D^2\om_\Phi\cdot(\e\pa_x,\pa_y)\f{1}{\pa_y\om^p}\cdot (\e\pa_x,\pa_y)\D^2\om_\Phi dxdy\\
%%%%%%%%%%%%%%%%%%%%%
&+\int_{\mathcal{S}}|\D^2\om_\Phi|^2\Big(\pa_x(\f{u^p}{\pa_y\om^p})+\pa_y(\f{v^p}{\pa_y\om^p})\Big)dxdy-\int_{\mathcal{S}}\big[\D^2,u^p\pa_x+v^p\pa_y\big]\om_\Phi ~\f{\D^2\om_\Phi}{\pa_y\om^p}dxdy\\
%%%%%%%%%%%
&-\int_{\mathcal{S}}\D^2(\pa_y\phi\pa_x\om^p)_\Phi~\f{\D^2\om_\Phi}{\pa_y\om^p}dxdy+\int_{\mathcal{S}}[\D^2,\pa_y\om^p]\pa_x\phi_\Phi~\f{\D^2\om_\Phi}{\pa_y\om^p}dxdy\\
%%%%%%%%%%%%%%%
&
+\int_\mathcal{S} \D^2 \pa_x\phi_\Phi\D^2\om_\Phi dxdy
-\int_\mathcal{S}\D^2\Big([e^{\Phi(t,D_x)},u^p\pa_x ] \om \Big)~ \f{\D^2\om_\Phi}{\pa_y\om^p}dxdy
\\
%%%%%%%%%%%%%%%%%%%
&-\int_\mathcal{S}\D^2\Big( [e^{\Phi(t,D_x)},v^p\pa_y ] \om  \Big)~ \f{\D^2\om_\Phi}{\pa_y\om^p}dxdy\\
%%%%%
&+\int_\mathcal{S}\D^2\Big( [e^{\Phi(t,D_x)},  \pa_y \om^p] \pa_x \phi   \Big)~ \f{\D^2\om_\Phi}{\pa_y\om^p}dxdy\\
%%%%%%%%%
&+\int_\mathcal{S}\D^2\Big(\pa_y (\mathcal{N}_u)_\Phi-\e^2\pa_x (\mathcal{N}_v)_\Phi\Big)~ \f{\D^2\om_\Phi}{\pa_y\om^p}dxdy\\
&+\int_\mathcal{S}\D^2(\e^2 f_1+f_2-C(t)\pa_x\om^p)_\Phi~ \f{\D^2\om_\Phi}{\pa_y\om^p}dxdy\\
=&T^0+\cdots T^{10}.
\end{align*}
The boundary term is zero due to artificial boundary condition $\om|_{y=0,1}=\Delta_\e\phi|_{y=0,1}=0.$ Integrating on $[0,t)$ with $t\leq T$ and using $\pa_y\om^p\geq c_0$, we obtain
\begin{align*}
\|\om(t)\|_{X^{2}}^2+2\la\int_0^t  \|\om\|_{X^{\f73}}^2ds+2\int_0^t\|\na_\e\om\|_{X^{2}}^2ds\leq C\int_0^t|T^0|+\cdots+|T^{10}|ds.
\end{align*}

\medskip

Now, we estimate $T^i, i=0,\cdots, 10$ one by one.

\underline{Estimate of $T^0$ and $T^1.$} Since $\pa_y\om^p\geq c_0>0$ and Lemma \ref{lem: u^p} imply
\begin{align*}
|(\e\pa_x,\pa_y)\f{1}{\pa_y\om^p}|\leq C,\quad |\pa_x(\f{u^p}{\pa_y\om^p})|+|\pa_y(\f{v^p}{\pa_y\om^p})|\leq C,
\end{align*}
it is easy to see 
\beno
|T^0|+|T^1|\leq& C\|\om\|_{X^2}(\|\na_\e\om\|_{X^2} +\|\om\|_{X^2}).
\eeno

\underline{Estimate of $T^2$ and $T^4.$} By using Lemma \ref{lem:com-S}, we get
\begin{align*}
&\|[\D^2,u^p\pa_x+v^p\pa_y\big]\om_\Phi\|_{L^2}\leq C(\|\om\|_{X^2}+\|\pa_y\om\|_{X^2}),\\
&\|[\D^2,\pa_y\om^p]\pa_x\phi_\Phi\|_{L^2}\leq C\|\phi\|_{X^2}\leq C\|\pa_y\phi\|_{X^2},
\end{align*}
where we used Poincar\'e inequality and $\phi|_{y=0,1}=0$ to ensure
\begin{align}\label{est: Poincare}
\|\phi\|_{X^r}\leq C\|\pa_y\phi\|_{X^r},\quad r\geq0.
\end{align}
  in the last step.
  
According to 
\begin{align}
\Delta_\e\phi=\om,\quad \phi|_{y=0,1}=0,
\end{align}
classical elliptic estimate and \eqref{est: Poincare} imply
\begin{align*}
\|\na_\e\phi\|_{L^2}^2\leq \|\om\|_{L^2}\|\phi\|_{L^2}\leq C\|\om\|_{L^2}\|\pa_y\phi\|_{L^2},
\end{align*}
which gives 
\begin{align}\label{est: na_ephi leq om}
\|\na_\e\phi\|_{X^r}\leq \|\om\|_{X^r},\quad r\geq0.
\end{align}
Therefore, it follows from $\pa_y\om^p\geq c_0>0$ to get
\begin{align*}
|T^2|+ |T^4|\leq C(\|\pa_y\om\|_{X^2}+\|\om\|_{X^2})\|\om\|_{X^2}.
\end{align*}

\underline{Estimate of $T_3.$} Using Lemma \ref{lem:product-Gev} and \eqref{est: na_ephi leq om}, it shows
\begin{align*}
|T^3|\leq&C\|\pa_y\phi\|_{X^2}\|\om\|_{X^2}\leq C\|\om\|_{X^2}^2.
\end{align*}

\underline{Estimate of $T^5.$} This term is the trouble term because it loses one tangential derivative. However, hydrostatic trick implies
\begin{align*}
T^5=&\int_\mathcal{S} \D^2 \pa_x\phi_\Phi\D^2\Delta_\e\phi_\Phi dxdy=-\int_\mathcal{S} \D^2 \pa_x\na_\e\phi_\Phi\D^2\na_\e\phi_\Phi dxdy\\
=&-\f12\int_\mathcal{S}  \pa_x|\D^2\na_\e\phi_\Phi|^2 dxdy=0,
\end{align*}
by using $\phi|_{y=0,1}=0.$

\underline{Estimate of $T^6,~T^7$ and $T^8.$} Let's estimate commutators by Lemma \ref{lem:com-Gev}. Since $\pa_y\om^p\geq c_0>0,$ we use Lemma \ref{lem:com-Gev} to ensure that
\begin{align*}
|T^6|\leq& C\|(u^p\pa_x\om)_\Phi-u^p\pa_x\om_\Phi \|_{H^{2-\f13,0}}\|\om\|_{X^\f73}\leq C\|\om\|_{X^{2+1-\f13-\f13}}\|\om\|_{X^\f73}=C\|\om\|_{X^\f73}^2,\\
|T^7|\leq& C\|\pa_y\om\|_{X^2}\|\om\|_{X^2},\\
|T^8|\leq&C\|(\pa_x\phi\pa_y\om^p)_\Phi-\pa_x\phi_\Phi\pa_y\om^p\|_{H^{2-\f13}}\|\om\|_{X^\f73}\leq C\|\pa_x\phi\|_{X^{2-\f13-\f13}}\|\om\|_{X^\f73}\leq C\|\phi\|_{X^\f73}\|\om\|_{X^\f73}\\
&\qquad\qquad\qquad\qquad\qquad\qquad\qquad\qquad\qquad\leq C\|\om\|_{X^\f73}^2.
\end{align*}
Here we use \eqref{est: Poincare} and \eqref{est: na_ephi leq om} in the last estimate.

\underline{Estimate of $T^9$ and $T^{10}.$} Integration by parts and boundary condition $\om|_{y=0,1}=0$ give that
\begin{align*}
|T^9|=&\int_\mathcal{S}\D^2(\mathcal{N}_u,\e \mathcal{N}_v)_\Phi\cdot \na_\e\Big(\f{\D^2\om_\Phi}{\pa_y\om^p}\Big)dxdy\\
\leq&C\|\mathcal{N}_u,\e \mathcal{N}_v\|_{X^2}(\|\om\|_{X^2}+\|\na_\e\om\|_{X^2}). 
\end{align*}

On the other hand, using H\"older inequality, we get
\begin{align*}
|T^{10}|\leq&C\|\e^2 f_1+f_2-C(t)\pa_x\om^p\|_{X^{\f53}}\|\om\|_{X^{\f73}}.
\end{align*}

Collecting $T^0-T^{10}$ together, we finally obtain
\begin{align*}
\int_0^t|T^0|+\cdots+|T^{10}|ds\leq& C\int_0^t\|\mathcal{N}_u,\e \mathcal{N}_v\|_{X^2}(\|\na_\e\om\|_{X^2}+\|\om\|_{X^2})\\
&+\|\om\|_{X^\f73}(\|\om\|_{X^\f73}+\|\na_\e\om\|_{X^2}+\|\e^2 f_1+f_2-C(t)\pa_x\om^p\|_{X^\f53})ds\\
\leq&\f1{10}\int_0^t\|\na_\e\om\|_{X^2}^2ds+C\int_0^t\|\mathcal{N}_u,\e \mathcal{N}_v\|_{X^2}^2ds\\
&+(C+\f{\la}{4})\int_0^t\|\om\|_{X^\f73}^2ds+\f{C}{\la}\int_0^t\|\e^2 f_1+f_2-C(t)\pa_x\om^p\|_{X^{\f53}}^2 ds.
\end{align*}
Taking $\la$ large enough , we deduce 
\begin{align}\label{est: om-good-1}
&\|\om(t)\|_{X^{2}}^2+\la\int_0^t  \|\om\|_{X^{\f73}}^2ds+\int_0^t\|\na_\e\om\|_{X^{2}}^2ds\\
\nonumber
\leq& C\int_0^t(\|\mathcal{N}_u,\e \mathcal{N}_v\|_{X^2}^2ds+\f{C}{\la}\int_0^t\|\e^2 f_1+f_2-C(t)\pa_x\om^p\|_{X^{\f53}}^2 ds.
\end{align}

On the other hand,  \eqref{est: na_ephi leq om} gives
\beno
\|\na_\e\phi\|_{X^\f73}\leq C\|\om\|_{X^\f73}.
\eeno
Calderon-Zygmund inequality and Gagliardo-Nirenberg inequality \eqref{equality: GN} imply
\begin{align*}
|\na_\e\phi|_{y=0,1}|_{X^\f73}\leq C\|\na_\e\phi\|_{X^\f73}^\f12(\|\na_\e\phi\|_{X^\f73}^\f12+\|\na_\e\pa_y\phi\|_{X^\f73}^\f12)\leq C\|\om\|_{X^\f73}.
\end{align*}
Along with \eqref{est: na_ephi leq om} and \eqref{est: om-good-1}, we get the desired result.

\end{proof}

\section{Construction of the boundary corrector: Proof of Proposition \ref{pro: phi_{bc}}}

In the previous section, we construct a solution to the Orr-Sommerfeld equation with artificial boundary conditions: we replace condition $\pa_y\phi|_{y=0,1}=0$ by $\Delta_\e\phi|_{y=0,1}=0.$ To go back to the original system, we need to correct Neumann condition. 
Thus, we define $\phi_{bc}$ satisfies the following system:
 \begin{align}\label{eq: phi-1-1}
\left\{
\begin{aligned}
&(\pa_t -\Delta_\e)\Delta_\e\phi_{bc}+u^p\pa_x \Delta_\e\phi_{bc}+v^p\pa_y \Delta_\e\phi_{bc}+\pa_y\phi_{bc}\pa_x \om^p -\pa_x\phi_{bc}\pa_y \om^p=0,\\
&\phi_{bc}|_{y=0,1}=0,\quad \pa_y\phi_{bc}|_{y=0,1}=- \pa_y\phi_{slip}|_{y=0,1}+C(t),\\
&\phi|_{t=0}=0,
\end{aligned}
\right.
\end{align}
%where $x\in\mathbb{T}$ and $y\in(0,1).$  For convenience, we denote 
%\begin{align}\label{def: (m^0, m^1)}
%m^0=- \pa_y\phi_{slip}|_{y=0}+C(t), \quad m^1=- \pa_y\phi_{slip}|_{y=1}+C(t).
%\end{align}

To estimate $\phi_{bc}$, we use the  following decomposition:
\beno
\phi_{bc}=\phi_{bc, S}+\phi_{bc, T}+\phi_{bc, R},
\eeno
The definitions and estimates of $\phi_{bc, S}, \phi_{bc, T}$ and $\phi_{bc, R}$ are given in the following subsections. 
 
\subsection{The estimates of $\phi_{bc, S}$: Stokes equation } In this subsection, we  deal with $\phi_{bc, S}$. 

Because of two boundary $y=0$ and $y=1$, we define 
\beno
\phi_{bc, S} = \phi^0_{bc, S}+\phi^1_{bc, S},
\eeno
where $\phi^0_{bc, S}$ satisfies the following Stokes equation:
\begin{align}\label{eq: Stokes}
\left\{
\begin{aligned}
&(\pa_t -\Delta_\e)\Delta_\e\phi^0_{bc, S}=0,\quad (x, y)\in \mathbb{T}\times (0,+\infty)\\
&\phi^0_{bc, S}|_{y=0}=0,\quad \pa_y\phi^0_{bc, S}|_{y=0}=h^0,\\
&\phi^0_{bc, S}|_{t=0}=0,
\end{aligned}
\right.
\end{align}
and $\phi^1_{bc, S}$ satisfies the following Stokes equation
\begin{align}\label{eq: Stokes-1}
\left\{
\begin{aligned}
&(\pa_t -\Delta_\e)\Delta_\e\phi^1_{bc, S}=0,\quad (x, y)\in \mathbb{T}\times (-\infty,1)\\
&\phi^1_{bc, S}|_{y=1}=0,\quad \pa_y\phi^1_{bc, S}|_{y=1}=h^1,\\
&\phi^1_{bc, S}|_{t=0}=0,
\end{aligned}
\right.
\end{align}
where $t\in[0,T]$. Here $(h^0,h^1)$ is a given boundary data satisfying $(h^0(t), h^1(t))=0$ for $t=0$ and $t\geq T.$  Here, we point out that $h^i$ is defined by
\beno
h^i=\mathcal{A}(- \pa_y\phi_{slip}|_{y=0,1}+C(t)),
\eeno
where the operator $\mathcal{A}$ is a zero-order operator which is defined later.

In the following, we only give the process  for $\phi^0_{bc, S}.$  The case of $\phi^1_{bc, S}$ is almost the same and we leave details to readers. 

At first, we give zero extension of $\phi^0_{bc, S}$ and $h^0$ with $t\leq 0$ such that we can take Fourier transform in $t.$ Let $\widehat{\phi^0_{bc, S}}=\widehat{\phi^0_{bc, S}}(\zeta, k, y)$ be the Fourier transform of $\phi^0_{bc, S}$ on $x$ and $t$. Then $\widehat{(\phi^0_{bc, S})_\Phi}$ satisfies the ODE:
\begin{align}\label{eq: phi_0-Fourier}
\left\{
\begin{aligned}
&-(\pa_y^2-\e^2|k|^2)^2\widehat{(\phi^0_{bc, S})_\Phi}+(i\zeta+\la\k^\f23)(\pa_y^2-\e^2|k|^2)\widehat{(\phi^0_{bc, S})_\Phi}=0,\quad y>0,\\
&\widehat{(\phi^0_{bc, S})_\Phi}|_{y=0}=0,\quad \pa_y\widehat{(\phi^0_{bc, S})_\Phi}|_{y=0}=\widehat{h^0_\Phi},
\end{aligned}
\right.
\end{align}
where $\zeta\in\mathbb{R}$ and $k\in\mathbb{Z}.$ Assuming the decay of $(|k|\phi^0_{bc, S},\pa_y\phi^0_{bc, S})$ and the boundedness of $\pa_y\phi^0_{bc, S},$ we obtain the formula:
\begin{align}
\widehat{(\phi^0_{bc, S})_\Phi}(\zeta, k, y)=&-\f{e^{-\gamma y}-e^{-\e |k|y}}{\gamma-\e|k|}\widehat{h^0_\Phi}(\zeta, k),\quad y>0\label{formula: phi^0}\\
\gamma=\gamma(\zeta,k,\e,\la)=&\sqrt{\e^2|k|^2+\la\k^\f23+i\zeta},\label{def: gamma}
\end{align}
where the square root is taken so that the real part is positive, and it follows that
\begin{align}\label{relation: gamma}
\e|k|,~\la^\f12 \k^\f13\leq\sqrt{\e^2|k|^2+\la \k^\f23}\leq \Re(\gamma)\leq |\gamma|\leq 2\Re(\gamma).
\end{align}
This inequality will be used frequently. It is easy to calculate that
\begin{align}
\pa_y\widehat{(\phi^0_{bc, S})_\Phi}=&-e^{-\gamma y}\widehat{h^0_\Phi}-\e|k|\widehat{(\phi^0_{bc, S})_\Phi},\label{formula: pa_y phi^0}\\
(\pa_y^2-\e^2|k|^2)\widehat{(\phi^0_{bc, S})_\Phi}=&(\gamma+\e|k|)e^{-\gamma y}\widehat{h^0_\Phi}.\label{formula: om^0}
\end{align}
The formula \eqref{formula: pa_y phi^0} will be used in estimating velocity and \eqref{formula: om^0} will be used in estimating vorticity.	With the same process above, we get the formula for $\widehat{(\phi^1_{bc, S})_\Phi}:$
\begin{align}
\widehat{(\phi^1_{bc, S})_\Phi}(\zeta, k, y)=\f{e^{-\gamma(1-y)}-e^{-\e |k|(1-y)}}{\gamma-\e |k|}\widehat{h^1_\Phi}(\zeta,k),\quad y<1,\label{formula: phi^1}
\end{align}
with $\gamma$ given in \eqref{def: gamma}. It is easy to see
\begin{align}
\pa_y\widehat{(\phi^1_{bc, S})_\Phi}=&e^{-\gamma (1-y)}\widehat{h^1_\Phi}+\e|k|\widehat{(\phi^1_{bc, S})_\Phi},\label{formula: pa_y phi^1}\\
(\pa_y^2-\e^2|k|^2)\widehat{(\phi^1_{bc, S})_\Phi}=&-(\gamma+\e|k|)e^{-\gamma (1-y)}\widehat{h^1_\Phi}.\label{formula: om^1}
\end{align}

\begin{remark}\label{rmk: difference}
For $\e=0$ in \eqref{eq: Stokes}, $\tri_0=\pa_y^2.$
\begin{align}
\widehat{(\phi^0_{bc, S})_\Phi}(\zeta, k, y)=-\f{\widehat{h^0_\Phi}}{\gamma_0}(e^{-\gamma_0 y}-1),\quad \gamma_0=\sqrt{\la\k^\f23+i\zeta}
\end{align}
solves \eqref{eq: phi_0-Fourier} with $\e=0$ and $\widehat{(\phi^0_{bc, S})_\Phi}$ holds $\lim_{y\to +\infty}=\f{\widehat{h^0_\Phi}}{\gamma_0}.$ Though $\widehat{(\phi^0_{bc, S})_\Phi}$ don't tend to zero as $y$ tends to infinity, the solution $\widehat{(\phi^0_{bc, S})_\Phi}$ is only used to correct boundary condition near $y=0$ and we don't care about its value at infinity. It is easy to deduce
\begin{align}
\pa_y\widehat{(\phi^0_{bc, S})_\Phi}=\widehat{h^0_\Phi}e^{-\gamma_0 y},\quad \pa_y^2\widehat{(\phi^0_{bc, S})_\Phi}=-\gamma_0\widehat{h^0_\Phi}e^{-\gamma_0 y},
\end{align}
and we find these two term are decay to zero as $y$ tends to infinity. By the same method, we can get another solution near $y=1:$
\begin{align}
\widehat{(\phi^1_{bc, S})_\Phi}(\zeta, k, y)=\f{\widehat{h^0_\Phi}}{\gamma_0}(e^{-\gamma_0 (1-y)}-1).
\end{align}
These constructions are main difference between $\e=0$ and $\e\neq0,$ but they enjoy the same properties stated below.
\end{remark}

\begin{lemma}\label{lem: na_e phi}
Let $\phi^i_{bc,S}$ be solution of \eqref{eq: Stokes}. It holds that
\begin{align}\label{est: na_e phi^0,1}
\sum_{k\in\mathbb{Z}}\|(\e|k|\widehat{(\phi^i_{bc, S})_\Phi},\pa_y\widehat{(\phi^i_{bc, S})_\Phi})\|_{L^2_{\zeta,y}} \leq& \f{C}{\la^\f14}\sum_{k\in\mathbb{Z}}\|\k^{-\f16}\widehat{h^i_\Phi}\|_{L^2_{\zeta}} ,
\end{align}
 where $i=0,1$ and $L^2_{\zeta, y}=l^2_{\zeta}(L^2_{y}(0,+\infty))$ for $i=0$ and $L^2_{\zeta, y}=l^2_{\zeta}(L^2_y(-\infty,1))$ for $i=1.$
 
 It is also holds that
\begin{align}\label{est: phi^0,1}
\sum_{k\in\mathbb{Z}}\|k\widehat{(\phi^i_{bc, S})_\Phi}\|_{L^2_{\zeta,y}} \leq& \f{C}{\la^{\f12}}\|k\k^{-\f13}  \widehat{h^i_\Phi}\|_{L^2_\zeta} ,
\end{align}
where $i=0,1$ and $L^2_{\zeta, y}=l^2_{\zeta}(L^2_{y}(0,1)).$
\end{lemma}

\begin{proof}
We only give the proof for $i=0.$ The case $i=1$ is almost the same and we omit details to readers.

\eqref{est: na_e phi^0,1} follows from \eqref{formula: phi^0}, \eqref{formula: pa_y phi^0} and the Plancherel theorem ,  by observing the estimate for multipliers
\begin{align}
\|e^{-\Re(\gamma)y}\|_{L^2_{y}(0,\infty)}\leq& \f{C}{\la^\f14\k^{\f16}},\label{na_e phi^0-1}\\
\|\e|k|\cdot e^{-\e|k|y}\cdot |\f{1-e^{-(\gamma-\e|k|)y}}{\gamma-\e|k|}|\|_{L^2_y(0,+\infty)}\leq &\f{C}{\la^\f14\k^{\f16}}.\label{na_e phi^0-2}
\end{align}
The estimate \eqref{na_e phi^0-1}	is a direct consequence of 
\begin{align}
\Re(\gamma)\geq \f{1}{\la^\f12\k^\f13}.
\end{align}
For \eqref{na_e phi^0-2}, we divide it into two cases: 1. $\e|k|\leq \f12\la^\f12\k^{\f13}$, and 2. $\e|k|\geq \f12\la^\f12\k^{\f13}.$ In case 1,
\begin{align*}
\Big|\gamma-\e|k|\Big|\geq \f{\e|k|+\la^\f12\k^\f13}{C},
\end{align*}
which implies
\begin{align*}
\|\e|k|\cdot e^{-\e|k|y}\cdot|\f{1-e^{-(\gamma-\e|k|)y}}{\gamma-\e|k|}|\|_{L^2_y(0,+\infty)}\leq &\f{C}{\e|k|+\la^\f12\k^\f13}\|\e|k|e^{-\e|k|y}\|_{L^2_y(0,+\infty)}\\
\leq& C\f{(\e|k|)^\f12}{\e|k|+\la^\f12\k^\f13}\leq \f{C}{\la^\f14\k^{\f16}}.
\end{align*}
 In case 2, we use the bound
 \begin{align*}
 |\f{1-e^{-z}}{z}|\leq C,
 \end{align*}
for $\Re(z)>0,$ which implies that
\begin{align*}
\|\e|k|\cdot e^{-\e|k|y}\cdot |\f{1-e^{-(\gamma-\e|k|)y}}{\gamma-\e|k|}|\|_{L^2_y(0,+\infty)}\leq &\|y\e|k|e^{-\e|k|y}\|_{L^2_y(0,+\infty)}\leq \f{C}{(\e|k|)^\f12}\leq \f{C}{\la^\f14\k^{\f16}}.
\end{align*}
Combining case 1-2 together, we complete \eqref{na_e phi^0-2}, which yields \eqref{est: na_e phi^0,1}. The estimate \eqref{est: phi^0,1} is proved by using \eqref{formula: phi^0}, Placherel theorem and 
\begin{align}\label{na_e phi^0-3}
\|e^{-\e|k|y}\cdot|\f{1-e^{-(\gamma-\e|k|)y}}{\gamma-\e|k|}|\|_{L^2_y(0,1)}\leq &\f{C}{\la^\f12\k^{\f13}}.
\end{align}
Indeed, note that the integral interval is $y\in(0,1)$ and  we also divide it into $\e|k|\leq \f12\la^\f12\k^{\f13}$ and  $\e|k|\geq \f12\la^\f12\k^{\f13}.$ When $\e|k|\geq \f12\la^\f12\k^{\f13},$ similar argument above gives that
\begin{align}
\|e^{-\e|k|y}\cdot |\f{1-e^{-(\gamma-\e|k|)y}}{\gamma-\e|k|}|\|_{L^2_y(0,1)}\leq \f{C}{\la^\f34\k^{\f12}}.
\end{align}
When $\e|k|\leq \f12\la^\f12\k^{\f13}$( with $\e|k|\ll 1$), we compute as
\begin{align*}
\|e^{-\e|k|y}\cdot |\f{1-e^{-(\gamma-\e|k|)y}}{\gamma-\e|k|}|\|_{L^2_y(0,1)}\leq C\|\f{1}{\e|k|+\la^\f12\k^\f13}\|_{L^2_y(0,1)}\leq \f{C}{\la^\f12\k^{\f13}}.
\end{align*}
The finite interval $(0,1)$ is essential here. Thus we complete this lemma.

\end{proof}

 In order to express clearly, we introduce norms related to  $y>0$ and $y<1$ respectively. For any function $f$, we define
\begin{align}
\|f\|_{X^r_i}=\|f_\Phi\|_{L^2_y(I_i;H^r_x)},
\end{align}
where $I_0=(0,+\infty)$ and $I_1=(-\infty,1).$ It is obvious to see $\|\cdot\|_{X^r}\leq\|\cdot\|_{X^r_i}$ for any $i=0,1.$ Using Lemma \ref{lem: na_e phi} above, we can deduce the estimate for $\na_\e\phi^i,$ where $i=0,1.$
\begin{proposition}\label{pro: na_e phi}
Let $\phi^i_{bc,S}$ be solution of \eqref{eq: Stokes}. It holds that
\begin{align}
\int_0^t\|\na_\e\phi^i_{bc,S}\|_{X^{\f73+\f16}_i}^2ds\leq \f{C}{\la^\f12}\int_0^t|h^i|_{X^\f73}^2ds,\\
\int_0^t\|\pa_x\phi^i_{bc,S}\|_{X^\f{5}{3}}^2ds\leq \f{C}{\la}\int_0^t|\pa_x h^i|_{X^\f43}^2ds.
\end{align}

\end{proposition}
\begin{proof}

The proof is done by using \eqref{est: na_e phi^0,1} and \eqref{est: phi^0,1}.

\end{proof}

Next, we give the estimate to boundary term $\phi^0_{bc,S}|_{y=1}$ and $\phi^1_{bc,S}|_{y=0.}$
\begin{lemma}\label{lem: phi^i|_y=1-i}
For any $M\geq0$ and $i=0,1,$ it holds that
\begin{align}\label{est: (e|k|)^M phi^i|_y=1-i}
\int_0^t|(\e\pa_x)^M\phi^i_{bc,S}|_{y=1-i}|_{X^{r+\f13}}^2ds\leq &\f{C}{\la}\int_0^t|h^i|_{X^{r}}^2ds,
\end{align}
and
\begin{align}\label{est: (e|k|)^M pa_y phi^i|_y=1-i}
\int_0^t|(\e\pa_x)^M\pa_y\phi^i_{bc,S}|_{y=1-i}|_{X^{r+\f13}}^2ds\leq &\f{C}{\la}\int_0^t|h^i|_{X^{r}}^2ds,
\end{align}
for any $r\geq0.$
\end{lemma}
\begin{proof}
We only give the proof for the case $i=0$, the case $i=1$ is similar and we omit details to readers. Taking $y=1$ in \eqref{formula: phi^0} and using 
\begin{align}\label{est: (e|k|)^M phi^0|_y=1-1}
\Big|(\e|k|)^M \cdot e^{-\e|k|} \cdot \f{e^{-(\gamma-\e|k|)}-1}{\gamma-\e|k|}\Big|\leq \f{C}{\la^\f12\k^\f13},
\end{align}
  we get
  \begin{align*}
  \int_0^t|(\e|k|)^M\phi^0_{bc,S}|_{y=1}|_{X^{r+\f13}}^2ds\leq &\f{C}{\la}\int_0^t|h^0|_{X^{r}}^2ds.
  \end{align*}

On the other hand, we refer to \eqref{formula: pa_y phi^0} and take $y=1$ in it  by noticing
\begin{align*}
\Big|e^{-\Re(\gamma)}(\e|k|)^M\Big|\leq &\Big|e^{-\f12\e|k|}(\e|k|)^Me^{-\f12\la^\f12\k^\f16}\Big|\leq Ce^{-\f12\la^\f12\k^\f13}\leq \f{C}{(\la\k^\f23)^{N/2}},
\end{align*}
for  any $N\geq0$, and combining with \eqref{est: (e|k|)^M phi^0|_y=1-1} to deduce
\begin{align}
\int_0^t|(\e\pa_x)^M\pa_y\phi^0_{bc,S}|_{y=1}|_{X^{r+\f13}}^2ds\leq &\f{C}{\la}\int_0^t|h^0|_{X^{r}}^2ds.
\end{align}
Thus, we finish our proof.
\end{proof}

\medskip

In the end of this subsection, we give some weight estimates of vorticity $\om^i_{bc,S}=\Delta_\e\phi^i_{bc,S}$. Denote
\begin{align}\label{def: varphi^i}
\varphi^0(y)=y,\quad \varphi^1(y)=1-y.
\end{align}

\begin{proposition}\label{pro: om^i}
It holds that
\begin{align}
|\widehat{(\om^i_{bc,S})_\Phi}(\zeta,k,y)|+|\varphi^i\pa_y \widehat{(\om^i_{bc,S})_\Phi}(\zeta,k,y)|\leq& C(|\gamma|+\e|k|)e^{-\Re(\gamma)\varphi^i}|\widehat{h^i_\Phi}(\zeta,k)|.
\end{align}
As a consequences, we get for $\theta'\in[-\f12,2]$
\begin{align*}
\int_0^t \|(\varphi^i)^{1+\theta'} \om^i_{bc,S}\|_{X^{\f73+\f13(\theta'+\f12)}_i}^2+\|(\varphi^i)^{2+\theta'}(\pa_y,\e|k|)\om^i_{bc,S}\|_{X^{\f73+\f13(\theta'+\f12)}_i}^2ds\leq \f{C}{\la^{\f12+\theta'}}\int_0^t |h^i|_{X^{\f73}}^2ds.
\end{align*}

\end{proposition}

\begin{proof}
The result is obtained by using formula \eqref{def: gamma},  \eqref{formula: om^0} and \eqref{formula: om^1}, the Plancherel theorem and by observing that multiplier $\varphi^i(y)$ gains $\f{1}{\la^\f12\k^\f13}$. More precisely,
\begin{align*}
\|(\varphi^i)^{1+m}|\gamma|e^{-\Re(\gamma)\varphi^i}\|_{L^2_y(I_i)}\leq (\f{C}{\la^\f12\k^\f13})^{m+\f12}.
\end{align*}
Thus we complete the proof.

\end{proof}

Based on the above proposition, we have more estimates on $\om^i_{bc,S}$:
\begin{proposition}\label{pro: Estimate for om^i-transport}
Let $\theta\in[0,2].$ It holds that
\begin{align}
&\int_0^t\|\varphi^i\Delta_\e\phi^i_{bc,S}\|_{X^{\f73+\f16}_i}^2+ \|(\varphi^i)^2\pa_y\Delta_\e\phi^i_{bc,S}\|_{X^{\f73+\f16}_i}^2ds\leq \f{C}{\la^\f12}\int_0^t|h^i|_{X^\f73}^2ds,\label{est: om^i-1}\\
&\int_0^t\|\D^{\f{\theta}{3}-\f13}(\varphi^i)^{\theta+\f32}(\pa_x\Delta_\e\phi^i_{bc,S})\|_{X^2_i}^2ds\leq \f{C}{\la^{\theta+1}}
\int_0^t|h^i|_{X^\f73}^2ds,\label{est: om^i-2}\\
&\int_0^t\|\D^{\f{\theta}{3}-\f13}(\varphi^i)^{\theta+\f32}(\pa_y\Delta_\e\phi^i_{bc,S})\|_{X^2_i}^2ds\leq \f{C}{\la^{\theta}}\int_0^t|h^i|_{X^\f53}^2ds.\label{est: om^i-3}
\end{align}

\end{proposition}
\begin{proof}
\eqref{est: om^i-1} is a direct result of Lemma \ref{pro: om^i} by taking $\theta=0.$  It is easy to check
\begin{align*}
\|\k^{\f{\theta}{3}-\f13}(\varphi^i)^{\theta+\f32}k\widehat{\Delta_\e(\phi^i_{bc,S})_\Phi}\|_{L^2_y(I_i)}\leq C\f{\k^{\f23+\f{\theta}{3}}}{(\la^\f12\k^\f13)^{\theta+\f12+\f12}}|\widehat{h^i_\Phi}|=\f{C\k^\f13}{\la^{\f12(\theta+1)}}|\widehat{h^i_\Phi}|,
\end{align*}
by taking $\theta'=\theta+\f12$ in Lemma \ref{pro: om^i} and complete \eqref{est: om^i-2}. Similarly, we check
\begin{align*}
\|\k^{\f{\theta}{3}-\f13}(\varphi^i)^{\theta+\f32}\pa_y\widehat{\Delta_\e(\phi^i_{bc,S})_\Phi}\|_{L^2_y(I_i)}\leq \f{C\k^{\f{\theta}{3}-\f13}}{\la^\f{\theta}{2}\k^{\f{\theta}{3}}}|\widehat{h^i_\Phi}|
\leq \f{C}{\la^{\f{\theta}{2}}\k^\f13}|\widehat{h^i_\Phi}|,
\end{align*}
by taking $\theta'=\theta-\f12$ in Lemma \ref{pro: om^i} to complete \eqref{est: om^i-3}.
\end{proof}

\subsection{The estimates of $\phi_{bc, T}$: Vorticity transport estimate.}

 $\phi_{bc, T}$   is defined by
\beno
\phi_{bc, T}=\phi^0_{bc, T}+\phi^1_{bc, T},
\eeno
where $\phi^0_{bc, T}$ is defined by
\begin{align}\label{eq: om^i-transport}
\left\{
\begin{aligned}
&(\pa_t-\Delta_\e)\Delta_\e\phi^0_{bc, T}+u^p\pa_x\Delta_\e\phi^0_{bc, T}+v^p\pa_y\Delta_\e\phi^0_{bc, T}\\
&\quad \quad= -u^p\pa_x\Delta_\e \phi_{bc, S}^0-v^p\pa_y\Delta_\e \phi_{bc, S}^0\eqdef H^0,\quad (x,y)\in \mathbb{T}\times (0,+\infty)\\
&\phi^0_{bc, T}|_{y=0}=0,\quad \Delta_\e\phi^0_{bc, T}|_{y=0}=0,\quad \phi^0_{bc, T}|_{t=0}=0.
\end{aligned}
\right.
\end{align}
and $\phi^1_{bc, T}$ is defined by
\begin{align}\label{eq: om^i-transport-1}
\left\{
\begin{aligned}
&(\pa_t-\Delta_\e)\Delta_\e\phi^1_{bc, T}+u^p\pa_x\Delta_\e\phi^1_{bc, T}+v^p\pa_y\Delta_\e\phi^1_{bc, T}\\
&\quad \quad= -u^p\pa_x\Delta_\e \phi_{bc, S}^1-v^p\pa_y\Delta_\e \phi_{bc, S}^1\eqdef H^1,\quad (x,y)\in \mathbb{T}\times (-\infty, 1)\\
&\phi^1_{bc, T}|_{y=1}=0,\quad \Delta_\e\phi^1_{bc, T}|_{y=1}=0,\quad \phi^1_{bc, T}|_{t=0}=0.
\end{aligned}
\right.
\end{align}
We need to emphasize that we extend $(u^p, v^p)$ to $y\in \mathbb{R}$ by zero which means that $(u^p, v^p)=0$ when $y\in \mathbb{R}\setminus[0,1].$

Before we give the estimates of $\phi^i_{bc, T}$, using  Proposition \ref{pro: Estimate for om^i-transport} and $(u^p, v^p)|_{y=0,1}=0$ to get that
\begin{align}\label{est: H^i}
\int_0^t\|(\varphi^i)^{\f12+\theta}H^i\|_{X^{\f{5}{3}+\f{\theta}{3}}_i}^2ds\leq& C\int_0^t\|(\varphi^i)^{\f32+\theta}(\pa_x\Delta_\e\phi^i_{bc,S}+\pa_y\Delta_\e\phi^i_{bc,S})\|_{X^{\f{5}{3}+\f{\theta}{3}}_i}^2ds\\
\nonumber
\leq&C\int_0^t|h^i|_{X^\f73}^2ds,
\end{align}
 where $\theta=0,1,2$.
 
 \medskip
 
We are in the position to give the estimates of $\phi^i_{bc,T}$:
\begin{proposition}\label{pro: om^i-transport}
Let $\theta=0,1,2$ and $i=0,1$. There exists $\la_0>1$ and $0<T<\min\{T_p, \f{1}{2\lambda}\}$ such that for all $t\in[0,T]$, $\la\geq \la_0$, it holds that 
\begin{align*}
&\|(\varphi^i)^\theta\om^i_{bc,T}\|_{X^{\f{11}{6}+\f{\theta}3}_i}^2+\la\int_0^t\|(\varphi^i)^\theta\om^i_{bc,T}\|_{X^{\f{13}{6}+\f{\theta}3}_i}^2ds\\
&\qquad+\int_0^t\|(\varphi^i)^\theta\na_\e\om^i_{bc,T}\|_{X^{\f{11}{6}+\f{\theta}3}_i}^2ds\leq \f{C}{\la^\f12} \int_0^t|h^i|_{X^\f73}^2ds,
\end{align*}
where $\Delta_\e\phi^i_{bc,T}=\om^i_{bc,T}$ and $\varphi^i$ is given in \eqref{def: varphi^i}.
\end{proposition}

\begin{proof}
Acting $e^{\Phi(t,D_x)}$ on the first equation of \eqref{eq: om^i-transport}, we obtain 
\begin{align}\label{eq: om^i_Phi-transport}
(\pa_t+&\la\D^\f23-\Delta_\e)(\om^i_{bc,T})_\Phi+u^p\pa_x(\om^i_{bc,T})_\Phi+v^p\pa_y(\om^i_{bc,T})_\Phi\\
\nonumber
&+\Big((u^p\pa_x\om^i_{bc,T})_\Phi-u^p\pa_x(\om^i_{bc,T})_\Phi\Big)+\Big((v^p\pa_y\om^i_{bc,T})_\Phi-v^p\pa_y(\om^i_{bc,T})_\Phi\Big)=H^i_\Phi.
\end{align}
Then taking $L^2_y(I_i;H^{\f{11}{6}+\f{\theta}3}_x)$ inner product with $(\varphi^i)^{2\theta}(\om^i_{bc,T})_\Phi$, we get by using $\om^i_{bc,T}|_{y=i}=0,~\pa_xu^p+\pa_y v^p=0$ and integrating by parts that
\begin{align*}
\f12\f{d}{dt}&\|(\varphi^i)^\theta\om^i_{bc,T}\|_{X^{\f{11}{6}+\f{\theta}3}_i}^2+\la\|(\varphi^i)^\theta\om^i_{bc,T}\|_{X^{\f{13}{6}+\f{\theta}3}_i}^2+\|(\varphi^i)^\theta\na_\e\om^i_{bc,T}\|_{X^{\f{11}{6}+\f{\theta}3}_i}^2\\
=&-\int_{\mathcal{S}_i}(\varphi^i)^{2\theta}[\D^{\f{11}{6}+\f{\theta}3},u^p\pa_x+v^p\pa_y](\om^i_{bc,T})_\Phi~ \D^{\f{11}{6}+\f{\theta}3}(\om^i_{bc,T})_\Phi dxdy\\
&+\f12\int_{\mathcal{S}_i}\pa_y \big((\varphi^i)^{2\theta}\big) v^p|\D^{\f{11}{6}+\f{\theta}3}(\om^i_{bc,T})_\Phi|^2dxdy\\
&-\int_{\mathcal{S}_i}\D^{\f{11}{6}+\f{\theta}3}\Big((u^p\pa_x\om^i_{bc,T})_\Phi-u^p\pa_x(\om^i_{bc,T})_\Phi\Big)~\D^{\f{11}{6}+\f{\theta}3}(\om^i_{bc,T})_\Phi(\varphi^i)^{2\theta} dxdy\\
&-\int_{\mathcal{S}_i}\D^{\f{11}{6}+\f{\theta}3}\Big((v^p\pa_y\om^i_{bc,T})_\Phi-v^p\pa_y(\om^i_{bc,T})_\Phi\Big)~\D^{\f{11}{6}+\f{\theta}3}(\om^i_{bc,T})_\Phi(\varphi^i)^{2\theta} dxdy\\
&-\int_{\mathcal{S}_i} \pa_y \big((\varphi^i)^{2\theta}\big) \D^{\f{11}{6}+\f{\theta}3}\pa_y(\om^i_{bc,T})_\Phi~\D^{\f{11}{6}+\f{\theta}3}(\om^i_{bc,T})_\Phi dxdy\\
&+\int_{\mathcal{S}_i}\D^{\f{11}{6}+\f{\theta}3}H^i_\Phi~\D^{\f{11}{6}+\f{\theta}3}(\om^i_{bc,T})_\Phi(\varphi^i)^{2\theta} dxdy\\
=&I_1^i+\cdots+I_6^i,
\end{align*}
where $\mathcal{S}_i=\mathbb{T}\times I_i.$
Integrating on $[0,t)$ with $t\leq T$, we obtain
\begin{align}\label{est: om^i-transport}
\|(\varphi^i)^\theta\om^i_{bc,T}(t)\|_{X^{\f{11}{6}+\f{\theta}3}_i}^2+&2\la\int_0^t\|(\varphi^i)^\theta\om^i_{bc,T}\|_{X^{\f{13}{6}+\f{\theta}3}_i}^2ds+2\int_0^t\|(\varphi^i)^\theta\na_\e\om^i_{bc,T}\|_{X^{\f{11}{6}+\f{\theta}3}_i}^2ds\\
\nonumber
\leq&2\int_0^t|I_1^i|+\cdots+ |I_6^i| ds.
\end{align}

Now, we estimate $I_j^i,~j=1,\cdots,6$ term by term.

\underline{Estimate of $I_1^i.$} It follows from Lemma \ref{lem:com-S} that
\begin{align*}
\|[\D^{\f{11}{6}+\f{\theta}3},u^p\pa_x+v^p\pa_y](\om^i_{bc,T})_\Phi\|_{L^2_x}\leq C(\|(\om^i_{bc,T})_\Phi\|_{H^{\f{11}6+\f{\theta}{3}}_x}+\|\pa_y(\om^i_{bc,T})_\Phi\|_{H^{\f{11}{6}+\f{\theta}3}_x}),
\end{align*}
which deduces that
\begin{align*}
|I_1^i|\leq&C(\|(\varphi^i)^\theta\om^i_{bc,T}\|_{X^{\f{11}{6}+\f{\theta}3}_i}+\|(\varphi^i)^\theta\pa_y\om^i_{bc,T}\|_{X^{\f{11}{6}+\f{\theta}3}_i})\|(\varphi^i)^\theta\om^i_{bc,T}\|_{X^{\f{11}{6}+\f{\theta}3}_i}\\
\leq&\f1{10}\|(\varphi^i)^\theta\pa_y\om^i_{bc,T}\|_{X^{\f{11}{6}+\f{\theta}3}_i}^2+ C\|(\varphi^i)^\theta\om^i_{bc,T}\|_{X^{\f{11}{6}+\f{\theta}3}_i}^2.
\end{align*}

\underline{Estimate of $I_2^i.$} Thanks to 
\begin{align}\label{est: pa_y(varphi^i)}
|\pa_y \big((\varphi^i)^{2\theta}\big) v^p|=|2\theta(\varphi^i)' (\varphi^i)^{2\theta-1}v^p|\leq C\theta(\varphi^i)^{2\theta}|\f{v^p}{\varphi^i}|\leq C\theta(\varphi^i)^{2\theta},
\end{align}
by using $v^p|_{y=i}=0,$ it is obvious to see
\begin{align*}
|I_2^i|\leq C\theta\|(\varphi^i)^\theta\om^i_{bc,T}\|_{X^{\f{11}{6}+\f{\theta}3}_i}^2.
\end{align*}

\underline{Estimate of $I_3^i.$} Applying Lemma  \ref{lem:com-Gev}, we find
\begin{align*}
|I_3^i|\leq&\|(\varphi^i)^\theta\D^{\f{9}{6}+\f{\theta}3}\Big((u^p\pa_x\om^i_{bc,T})_\Phi-u^p\pa_x(\om^i_{bc,T})_\Phi\Big)\|_{L^2_y(I_i; L^2_x)}\|(\varphi^i)^\theta\om^i_{bc,T}\|_{X^{\f{13}{6}+\f{\theta}3}_i}\\
\leq&C\|(\varphi^i)^\theta\om^i_{bc,T}\|_{X^{\f{13}{6}+\f{\theta}3}_i}^2.
\end{align*}

\underline{Estimate of $I_4^i.$} Applying Lemma \ref{lem:product-Gev}, we get
\begin{align*}
|I_4^i|\leq& C\|(\varphi^i)^\theta\pa_y\om^i_{bc,T}\|_{X^{\f{11}{6}+\f{\theta}3}_i}\|(\varphi^i)^\theta\om^i_{bc,T}\|_{X^{\f{11}{6}+\f{\theta}3}_i}\\
\leq& \f1{10}\|(\varphi^i)^\theta\pa_y\om^i_{bc,T}\|_{X^{\f{11}{6}+\f{\theta}3}_i}^2+C\|(\varphi^i)^\theta\om^i_{bc,T}\|_{X^{\f{11}{6}+\f{\theta}3}_i}^2.
\end{align*}

\underline{Estimate of $I_5^i.$} By the fact
\begin{align*}
\pa_y \big((\varphi^i)^{2\theta}\big)=2\theta (\varphi^i)^{2\theta-1},
\end{align*}
 we have
\begin{align*}
|I_5^i|\leq& C\theta\|(\varphi^i)^\theta\pa_y\om^i_{bc,T}\|_{X^{\f{11}{6}+\f{\theta}3}_i}\|(\varphi^i)^{\theta-1}\om^i_{bc,T}\|_{X^{\f{11}{6}+\f{\theta}3}_i}\\
\leq&\f1{10}\|(\varphi^i)^\theta\pa_y\om^i_{bc,T}\|_{X^{\f{11}{6}+\f{\theta}3}_i}^2+C\theta^2\|(\varphi^i)^{\theta-1}\om^i_{bc,T}\|_{X^{\f{11}{6}+\f{\theta}3}_i}^2.
\end{align*}

\underline{Estimate of $I_6^i.$}
It follows from
\begin{align*}
\|\D^{2+\f{\theta}3}(\om^i_{bc,T})_\Phi(\varphi^i)^{\theta-\f12}\|_{L^2_y(I_i; L^2_x)}\leq&\|\D^{\f{13}{6}+\f{\theta}3}(\om^i_{bc,T})_\Phi(\varphi^i)^{\theta}\|_{L^2_y(I_i; L^2_x)}^\f12\|\D^{\f{11}{6}+\f{\theta}3}(\om^i_{bc,T})_\Phi(\varphi^i)^{\theta-1}\|_{L^2_y(I_i; L^2_x)}^\f12\\
\leq&\|(\varphi^i)^{\theta}\om^i_{bc,T}\|_{X^{\f{13}{6}+\f{\theta}3}_i}^\f12\|(\varphi^i)^{\theta-1}\om^i_{bc,T}\|_{X^{\f{11}{6}+\f{\theta}3}_i}^\f12,
\end{align*}
for $\theta=1,2$ and 
\begin{align*}
\|\D^{2}(\om^i_{bc,T})_\Phi(\varphi^i)^{-\f12}\|_{L^2_y(I_i; L^2_x)}\leq&\|\D^{\f{13}{6}}(\om^i_{bc,T})_\Phi\|_{L^2_y(I_i; L^2_x)}^\f12\|\D^{\f{11}{6}}(\om^i_{bc,T})_\Phi(\varphi^i)^{-1}\|_{L^2_y(I_i; L^2_x)}^\f12\\
\leq& C\|\om^i_{bc,T}\|_{X^{\f{13}{6}}_i}^\f12 \|\pa_y\om^i_{bc,T}\|_{X^{\f{11}{6}}_i}^\f12,
\end{align*}
by using Hardy inequality for $\theta=0.$ Therefore, we get for $\theta=0,1,2$ that
\begin{align*}
|I_6^i|\leq&C\|(\varphi^i)^{\f12+\theta}H^i\|_{X^{\f{5}{3}+\f{\theta}{3}}_i}\times \|(\varphi^i)^{\theta}\om^i_{bc, T}\|_{X^{\f{13}{6}+\f{\theta}3}_i}^\f12
\times
\left\{
\begin{aligned}
\|(\varphi^i)^{\theta-1}\om^i_{bc,T}\|_{X^{\f{11}{6}+\f{\theta}3}_i}^\f12,\quad \theta=1,2,\\
 \|\pa_y\om^i_{bc,T}\|_{X^{\f{11}{6}}_i}^\f12,\quad \theta=0,
\end{aligned}
\right.
\\
\leq&\f1{10}\|\pa_y\om^i_{bc,T}\|_{X^{\f{11}{6}}_i}^2+\f{\la}{4}\Big(\|(\varphi^i)^\theta\om^i_{bc,T}\|_{X^{\f{13}{6}+\f{\theta}3}_i}^2+\f{\theta^2} 2\|(\varphi^i)^{\theta-1}\om^i_{bc,T}\|_{X^{\f{11}{6}+\f{\theta}3}_i}^2\Big)+\f{C}{\la^\f12}\|(\varphi^i)^{\f12+\theta}H^i\|_{X^{\f{5}{3}+\f{\theta}{3}}_i}^2.
\end{align*}

Putting $I_1^i-I_6^i$ together, we have
\begin{align*}
|I_1^i|+\cdots+|I_6^i|\leq&\f1{2}\|(\varphi^i)^\theta\pa_y\om^i_{bc,T}\|_{X^{\f{11}{6}+\f{\theta}3}_i}^2+(C+\f{\la}{4})\|(\varphi^i)^\theta\om^i_{bc,T}\|_{X^{\f{13}{6}+\f{\theta}3}_i}^2+\f{\la}{8}\theta^2\|(\varphi^i)^{\theta-1}\om^i_{bc,T}\|_{X^{\f{11}{6}+\f{\theta}3}_i}^2\\
&+\f{C}{\la^\f12}\|(\varphi^i)^{\f12+\theta}H^i\|_{X^{\f{5}{3}+\f{\theta}{3}}_i}^2.
\end{align*}
Then we insert them into \eqref{est: om^i-transport} and take $\la$ large enough to obtain
\begin{align*}
\|(\varphi^i)^\theta\om^i_{bc,T}(t)\|_{X^{\f{11}{6}+\f{\theta}3}}^2+&\f32\la\int_0^t\|(\varphi^i)^\theta\om^i_{bc,T}\|_{X^{\f{13}{6}+\f{\theta}3}}^2ds+\int_0^t\|(\varphi^i)^\theta\na_\e\om^i_{bc,T}\|_{X^{\f{11}{6}+\f{\theta}3}}^2ds\\
\leq&\f{C}{\la^\f12}\int_0^t\|(\varphi^i)^{\f12+\theta}H^i\|_{X^{\f{5}{3}+\f{\theta}{3}}}^2+\f{\la}{8}\int_0^t\theta^2\|(\varphi^i)^{\theta-1}\om^i_{bc,T}\|_{X^{\f{11}{6}+\f{\theta}3}}^2ds\\
\leq&\f{C}{\la^\f12}\int_0^t|h^i|_{X^{\f{7}{3}}}^2ds+\f{\la}{8}\int_0^t\theta^2\|(\varphi^i)^{\theta-1}\om^i_{bc,T}\|_{X^{\f{11}{6}+\f{\theta}3}}^2ds,
\end{align*}
where we use \eqref{est: H^i} in the last step.

All we left is to control the last term of the above inequality. For that, we rewrite it as following:
\beno
\f{\la}{8}\int_0^t\theta^2\|(\varphi^i)^{\theta-1}\om^i_{bc,T}\|_{X^{\f{11}{6}+\f{\theta}3}}^2ds &\leq& \f{\la}{2} \sum_{\theta=1}^2\int_0^t \|(\varphi^i)^{\theta-1}\om^i_{bc,T}\|_{X^{\f{13}{6}+\f{\theta-1}3}}^2ds\\
&=&\f{\la}{2} \sum_{\theta=0}^1\int_0^t \|(\varphi^i)^{\theta}\om^i_{bc,T}\|_{X^{\f{13}{6}+\f{\theta}3}}^2ds.
\eeno

Combing all the above estimates, we get the desired results. 

\end{proof}

\medskip

%\begin{corollary}\label{cor: om^i-transport}
%Let $i=0,1,$
%there holds that
%\begin{align}\label{est: om^i-transport-1}
%&\sum_{\theta=0}^2\|(\varphi^i)^\theta\om^i_{bc,T}\|_{X^{\f{11}{6}+\f{\theta}3}_i}^2+\la\sum_{\theta=0}^2\int_0^t\|(\varphi^i)^\theta\om^i_{bc,T}\|_{X^{\f{13}{6}+\f{\theta}3}_i}^2ds\\
%\nonumber
%&\qquad +\sum_{\theta=0}^2\int_0^t\|(\varphi^i)^\theta\na_\e\om^i_{bc,T}\|_{X^{\f{11}{6}+\f{\theta}3}_i}^2ds\leq\f{C}{\la^\f12}\int_0^t|h^i|_{X^\f73}^2ds,
%\end{align}
%and
%\begin{align}\label{est: na_e phi^i-transport}
%\int_0^t\|\na_\e\phi^i_{bc,T}\|_{X^{\f52}}^2ds+\int_0^t|\pa_y\phi^i_{bc,T}|_{y=0,1}|_{X^{\f73}}^2ds\leq &\f{C}{\la^\f32} \int_0^t|h^i|_{X^\f73}^2ds.
%\end{align}
% \end{corollary}

Based on estimates of $\om_{bc, T}^i$, we use the elliptic equation to get the estimates of $\phi^i_{bc,T}$.

\begin{corollary}\label{cor: om^i-transport}
There exists $\la_0>1$ and $0<T<\min\{T_p, \f{1}{2\lambda}\}$ such that for all $t\in[0,T]$, $\la\geq \la_0$, it holds that 
 \begin{align}\label{est: na_e phi^i-transport}
\int_0^t\Big(\|\na_\e\phi^i_{bc,T}\|_{X^{\f73+\f16}}^2 + |\pa_y\phi^i_{bc,T}|_{y=0,1}|_{X^{\f73}}^2 +  \|\pa_x\phi^i_{bc,T}\|_{X^\f{5}{3}}^2+|\phi^i_{bc,T}|_{y=1-i}|_{X^\f{8}{3}}^2 \Big)ds\leq &\f{C}{\la } \int_0^t|h^i|_{X^\f73}^2ds.
\end{align}
 \end{corollary}

\begin{proof}
 
Here, we only give the proof of the case $i=0$. The case $i=1$ is the same.

Recalling the elliptic equation 
\begin{align}\label{eq: phi^0, y>0}
\Delta_\e\phi^0_{bc,T}=\om^0_{bc,T},\quad \phi^0_{bc,T}|_{y=0}=0, 
\end{align}
for $y>0$. Then we take $X^{r}_0$ inner product with $\phi^0_{bc,T}$ and using Hardy inequality to get
\begin{align*}
\|\na_\e\phi^0_{bc,T}\|_{X^{r}_0}^2\leq \|\varphi^0\om^0_{bc,T}\|_{X^{r}_0}\|\f{\phi^0_{bc,T}}{\varphi^0}\|_{X^{r}_0}\leq C\|\varphi^0\om^0_{bc,T}\|_{X^{r}_0}\|\pa_y\phi^0_{bc,T}\|_{X^{r}_0},
\end{align*}
which implies
\begin{align}\label{est: elliptic-2}
\|\na_\e\phi^0_{bc,T}\|_{X^{r}_0}\leq C\|\varphi^0\om^0_{bc,T}\|_{X^{r}_0},
\end{align}
for $r\geq0.$
By Proposition \ref{pro: om^i-transport}, we get
\begin{align}\label{est: na_e phi^i-transport-1}
\int_0^t\|\na_\e\phi^0_{bc,T}\|_{X^{\f52}_0}^2ds\leq&\int_0^t\|\varphi^0\om^0_{bc,T}\|_{X^{\f52}_0}^2ds\leq \f{C}{\la^\f32}\int_0^t|h^0|_{X^\f73}^2ds.
\end{align}

For the boundary term, using the interpolation inequality to get
\begin{align*}
|\pa_y\phi^0_{bc,T}|_{y=0,1}|_{X^{\f73}}\leq &C\|\pa_y\phi^0_{bc,T}\|_{X^{\f52}_0}^\f12\|\pa_y^2\phi^0_{bc,T}\|_{X^{\f{13}6}_0}^\f12\\
\leq&C\|\varphi^0\om^0_{bc,T}\|_{X^{\f52}_0}^\f12\|\om^0_{bc,T}\|_{X^{\f{13}6}_0}^\f12,
\end{align*}
where we use \eqref{est: elliptic-2} and Calderon-Zygmund inequality in the last step. Along with Proposition \ref{pro: om^i-transport}, we arrive at
\begin{align}\label{est: na_e phi^i-transport-2}
\int_0^t|\pa_y\phi^i_{bc,T}|_{y=0,1}|_{X^{\f73}}^2ds\leq&\f{C}{\la}\int_0^t|h^i|_{X^\f73}^2ds.
\end{align}

Next, we deal with the term $\|\pa_x\phi^0_{bc,T}\|_{X^\f{5}{3}}$.  Taking Fourier transform in $x$ to \eqref{eq: phi^0, y>0}, we write the solution
\begin{align}
\widehat{\phi}^0_{bc,T}(k,y)=\int_0^y e^{-\e|k|(y-y')}\int_{y'}^{+\infty}e^{-\e|k|(y''-y')}\widehat{\om}^0_{bc,T}(k, y'')d y'' dy'.
\end{align}
Then we have
\begin{align*}
|\widehat{\phi}^0_{bc,T}(k,y)|\leq &\int_0^y\int_{y'}^{+\infty}|\widehat{\om}^0_{bc,T}(k, y'')|d y'' dy'.
\end{align*}
Decomposing the integral $\int_0^y$ into $\int_0^{\min\{y,\k^{-\f13}\}}$ and $\int_{\min\{y,\k^{-\f13}\}}^y$, it follows from the H\"older inequality that
\begin{align*} 
\sup_{y\geq 0}|\widehat{\phi}^0_{bc,T}(k,y)|\leq& C\k^{-\f16}\|y\widehat{\om}^0_{bc,T}\|_{L^2_y(I_0)}+C\k^{\f16}\|y^2\widehat{\om}^0_{bc,T}\|_{L^2_y(I_0)}.
\end{align*}
We take summation $\sum_{k\in\mathbb{Z}}$ and use the Plancherel theorem to deduce 
\begin{align}\label{est: pointwise-phi^0}
\sup_{y\geq 0}\|\phi^0_{bc,T}(\cdot, y)\|_{L^2_x}\leq C\|\D^{-\f16}y\om^0_{bc,T}\|_{L^2_y(I_0; L^2_x)}+\|\D^{\f16}y^2\om^0_{bc,T}\|_{L^2_y(I_0; L^2_x)}.
\end{align}
Thus, we get that
\begin{align}\label{est: na_e phi^i-transport-3}
\int_0^t\Big(\|\pa_x\phi^i_{bc,T}\|_{X^\f{5}{3}}^2+|\phi^i_{bc,T}|_{y=1-i}|_{X^\f{8}{3}}^2\Big) ds\leq& C\int_0^t\|\varphi^i\om^i_{bc,T}\|_{X^{\f52}_i}^2ds+C\int_0^t\|(\varphi^i)^2\om^i_{bc,T}\|_{X^{\f{17}{6}}_i}^2ds\\
\nonumber
\leq&\f{C}{\la}\int_0^t|h^i|_{X^{\f73}}^2ds,
\end{align}

Collecting \eqref{est: na_e phi^i-transport-1}, \eqref{est: na_e phi^i-transport-2} and \eqref{est: na_e phi^i-transport-3} together, we get the corollary proved.

\end{proof}

\subsection{The estimates of $\phi_{bc, R}$: Full construction of boundary corrector}

All we left is the term $\phi_{bc, R}$. Like previous argument, we define
\beno
\phi_{bc, R}=\phi^0_{bc, R}+\phi^1_{bc, R},
\eeno
where $\phi^i_{bc, R}$ satisfies that
\begin{align}\label{eq: phi_2^0}
\left\{
\begin{aligned}
&(\pa_t-\Delta_\e)\Delta_\e \phi_{bc, R}^0+u^p\pa_x\Delta_\e \phi_{bc, R}^0+v^p\pa_y\Delta_\e \phi_{bc, R}^0+\pa_y\phi_{bc, R}^0\pa_x\om^p-\pa_x\phi_{bc, R}^0\pa_y\om^p\\
&\qquad\qquad\qquad\qquad=-\pa_y(\phi^0_{bc, S}+\phi^0_{bc, T})\pa_x\om^p+\pa_x(\phi_{bc, S}^0+\phi_{bc, T}^0)\pa_y\om^p,\quad t>0,~x\in\mathbb{T},~y\in(0,1),\\
&\qquad\qquad\qquad\qquad\eqdef G^0,\\
&\phi_{bc, R}^0|_{y=0}=0,\quad
\phi_{bc, R}^0|_{y=1}=-(\phi^0_{bc, S}+\phi^0_{bc, T})|_{y=1},\quad \Delta_\e\phi_{bc, R}^0|_{y=0,1}=0,\\
&\phi_{bc, R}^0|_{t=0}=0.
\end{aligned}
\right.
\end{align}
and 
\begin{align}\label{eq: phi_2^1}
\left\{
\begin{aligned}
&(\pa_t \Delta_\e)\Delta_\e \phi_{bc, R}^1+u^p\pa_x\Delta_\e \phi_{bc, R}^1+v^p\pa_y\Delta_\e \phi_{bc, R}^1+\pa_y\phi_{bc, R}^1\pa_x\om^p-\pa_x\phi_{bc, R}^1\pa_y\om^p\\
&\qquad\qquad\qquad\qquad=-\pa_y(\phi_{bc,S}^1+\phi_{bc,T}^1)\pa_x\om^p+\pa_x(\phi_{bc,S}^1+\phi_{bc,T}^1)\pa_y\om^p,\quad t>0,~x\in\mathbb{T},~y\in(0,1),\\
&\qquad\qquad\qquad\qquad\eqdef G^1,\\
&\phi_{bc, R}^1|_{y=0}=-(\phi^1_{bc, S}+\phi^1_{bc, T})|_{y=0},\quad
\phi_{bc, R}^1|_{y=1}=0,\quad \Delta_\e\phi_{bc, R}^1|_{y=0,1}=0,\\
&\phi_{bc, R}^1|_{t=0}=0.
\end{aligned}
\right.
\end{align}

For simplicity, denote $\om_{bc,R}^i=\Delta_\e \phi^i_{bc,R}$ who has the following relationship
\begin{align}\label{eq: tri_e phi_2^0=om_2^0}
\left\{
\begin{aligned}
&\Delta_\e \phi_{bc,R}^0=\om_{bc,R}^0,\\
&\phi_{bc,R}^0|_{y=0}=0,\quad
\phi_{bc,R}^0|_{y=1}=f^0,
\end{aligned}
\right.
\end{align}
and
\begin{align}\label{eq: tri_e phi_2^1=om_2^1}
\left\{
\begin{aligned}
&\Delta_\e \phi_{bc,R}^1=\om_{bc,R}^1,\\
&\phi_{bc,R}^1|_{y=0}=f^1,\quad
\phi_{bc,R}^1|_{y=1}=0,
\end{aligned}
\right.
\end{align}
where
\begin{align}
f^0=f^0(t,x)=&-(\phi^0_{bc,S}+\phi^0_{bc,T})|_{y=1},\label{def: f^0}\\
f^1=f^1(t,x)=&-(\phi^1_{bc,S}+\phi^1_{bc,R})|_{y=0}.\label{def: f^1}
\end{align}

In order to  homogenize boundary condition, we introduce
\begin{align}
\widetilde{\phi}_{bc,R}^0=\phi_{bc,R}^0+g^0,\quad g^0=y (\phi^0_{bc,S}+\phi^0_{bc,T}).\label{def: g^0}
\end{align}
where $\widetilde{\phi}_{bc,R}^0$ satisfies
 \begin{align}\label{eq: tilde phi_2^0}
 \left\{
 \begin{aligned}
 &\Delta_\e \widetilde{\phi}_{bc,R}^0=\om_{bc,R}^0+\Delta_\e g^0,\\
&  \widetilde{\phi}_{bc,R}^0|_{y=0,1}=0,
\end{aligned}
\right.
 \end{align}
 where 
\begin{align}
\Delta_\e g^0=y (\Delta_\e\phi^0_{bc,S}+\Delta_\e\phi^0_{bc,T})+2(\pa_y\phi^0_{bc,S}+\pa_y\phi^0_{bc,T}).\label{def: tri_e g^0}
\end{align}

Similarly, we introduce
\begin{align}
\widetilde{\phi}_{bc,R}^1=\phi_{bc,R}^1+g^1,\quad g^1=(1-y) (\phi^1_{bc,S}+\phi^1_{bc,T}).\label{def: g^1}
\end{align}
 and $\widetilde{\phi}_2^1$ satisfies
 \begin{align}\label{eq: tilde phi_2^1}
 \left\{
 \begin{aligned}
 &\Delta_\e \widetilde{\phi}_{bc,R}^1=\om_{bc,R}^1+\Delta_\e g^1,\\
&\widetilde{\phi}_{bc,R}^1|_{y=0,1}=0,
\end{aligned}
\right.
 \end{align}
 where 
\begin{align}
\Delta_\e g^1=(1-y)(\Delta_\e\phi^1_{bc,S}+\Delta_\e\phi^1_{bc,T})-2(\pa_y\phi^1_{bc,S}+\pa_y \phi^1_{bc,T}).\label{def: tri_e g^1}
\end{align}

\medskip

Firstly, we give some elliptic estimates.
\begin{lemma}\label{lem: na_e phi_2}
Let $(f^0, f^1),$ $(g^0, g^1)$ introduced in \eqref{def: f^0}-\eqref{def: f^1}, \eqref{def: g^0} and \eqref{def: g^1}. It holds that
\begin{align}\label{est: h-g}
\int_0^t\Big(|f^i|_{X^\f83}^2+ \|\na_\e g^i\|_{X^\f52}^2+\|\Delta_\e g^i\|_{X^\f52}^2\Big)ds\leq& \f{C}{\la^{\f12}}\int_0^t|h^{i}|_{X^{\f73}}^2ds  
\end{align}
for $i=0,1.$

Moreover,  $\phi_{bc,R}^i~(i=0,1)$  has the following estimates:
\begin{align}\label{est: na_e phi_2}
\int_0^t\|\na_\e\phi_{bc,R}^i\|_{X^\f73}^2ds\leq&\f{C}{\la^\f12}\int_0^t|h^i|_{X^{\f73}}^2ds+C\int_0^t\|\om_{bc,R}^i\|_{X^\f73}^2ds.
\end{align}
\end{lemma}
\begin{proof}
Here we only prove the case $i=0.$ The case $i=1$ is almost the same and we omit details to readers. 
We first give proof for $f^0$.  
By the definition of $f^0$, we get
\beno
\int_0^t|f^0|_{X^\f83}^2ds\leq \int_0^t|\phi^0_{bc,S}|_{y=1}|_{X^\f83}^2 ds+ \int_0^t|\phi^0_{bc,T}|_{y=1}|_{X^\f83}^2 ds \leq \f{C}{\la}\int_0^t|h^0|_{X^{\f73}}^2ds.
\eeno
where we used Lemma \ref{lem: phi^i|_y=1-i} and Corollary \ref{cor: om^i-transport}.

For $g^0$, by Corollary \ref{cor: om^i-transport}, Proposition \ref{pro: na_e phi},  we have
\beno
\int_0^t\|\na_\e g^0\|_{X^\f52}^2ds &\leq& \int_0^t\Big( \|\na_\e\phi_{bc,S}^0\|_{X^\f52} +  \|\na_\e\phi_{bc,T}^0\|_{X^\f52} +   \|\phi_{bc,S}^0\|_{X^\f83} + \|\phi_{bc,T}^0\|_{X^\f83}\Big)ds\\
&\leq &\f{C}{\la^\f12}\int_0^t|h^0|_{X^\f73}^2ds.
\eeno

 On one hand, 
using Proposition \ref{pro: na_e phi}, Proposition \ref{pro: om^i}, Corollary \ref{cor: om^i-transport}  and  Proposition \ref{pro: om^i-transport}, we get
\begin{align*}
&\int_0^t\Big(\|y\Delta_\e\phi^0_{bc,S}\|_{X^\f52}^2+\|\pa_y\phi^0_{bc,S}\|_{X^\f52}^2 +\|y\Delta_\e\phi^0_{bc,T}\|_{X^\f52}^2+\|\pa_y\phi^0_{bc,T}\|_{X^\f52}^2\Big)ds\leq \f{C}{\la^\f12}\int_0^t|h^0|_{X^\f73}^2ds,
\end{align*}
which implies that 
\begin{align*}
\int_0^t \|\Delta_\e g^0\|_{X^\f52}^2ds\leq \f{C}{\la^\f12}\int_0^t|h^0|_{X^\f73}^2ds.
\end{align*}

At last, we prove \eqref{est: na_e phi_2}.
Taking $X^{\f73}$ inner product with $\widetilde{\phi}_{bc,R}^0$ toward \eqref{eq: tilde phi_2^0}, we use integration by parts and then integrate time from $0$ to $t$ that
\begin{align*}
\int_0^t\|\na_\e\widetilde{\phi}_{bc,R}^0\|_{X^{\f73}}^2ds=-\int_0^t\langle \om_{bc,R}^0,\widetilde{\phi}_{bc,R}^0\rangle_{X^\f73}ds+\int_0^t\langle \Delta_\e g^0,\widetilde{\phi}_{bc,R}^0\rangle_{X^\f73}ds.
\end{align*}
Due to $ \widetilde{\phi}_{bc,R}^0|_{y=0,1}=0,$ we use Poincar\'e inequality to imply
\begin{align}
\int_0^t\langle \om_{bc,R}^0,\widetilde{\phi}_{bc,R}^0\rangle_{X^\f73}ds\leq&\f1{10}\int_0^t\|\pa_y\widetilde{\phi}_{bc,R}^0\|_{X^\f73}ds+ C\int_0^t\|\om_{bc,R}^0\|_{X^\f73}^2ds.\label{est: na_e phi-1}
\end{align}

According to \eqref{est: h-g}, we get
\begin{align}
\int_0^t\langle \Delta_\e g^0,\widetilde{\phi}_{bc,R}^0\rangle_{X^\f73}ds\leq&\int_0^t\|\Delta_\e g^0\|_{X^\f73}\|\widetilde{\phi}_{bc,R}^0\|_{X^\f73} ds\label{est: na_e phi-2}\\
\nonumber
\leq& \f1{10}\int_0^t\|\pa_y\widetilde{\phi}_{bc,R}^0\|_{X^\f73} ^2 ds+\f{C}{\la^\f12}\int_0^t|h^0|_{X^\f73}^2ds.
\end{align}

Combining \eqref{est: na_e phi-1} and \eqref{est: na_e phi-2}, it deduces
\begin{align*}
\int_0^t\|\na_\e\widetilde{\phi}_{bc,R}^0\|_{X^\f73}^2ds\leq&\f{C}{\la^\f12}\int_0^t|h^0|_{X^{\f73}}^2ds+C\int_0^t\|\om_{bc,R}^0\|_{X^\f73}^2ds.
\end{align*}
Bringing $\na_\e\phi_{bc,R}^0=\na_\e\widetilde{\phi}_{bc,R}^0-\na_\e g^0$ into above inequality, we obtain
\begin{align*}
\int_0^t\|\na_\e\phi_{bc,R}^0\|_{X^\f73}^2ds\leq &\int_0^t\|\na_\e\widetilde{\phi}_{bc,R}^0\|_{X^\f73}^2ds+\int_0^t\|\na_\e g^0\|_{X^\f73}^2ds\\
\leq&\f{C}{\la^\f12}\int_0^t|h^0|_{X^{\f73}}^2ds+C\int_0^t\|\om_{bc,R}^0\|_{X^\f73}^2ds.
\end{align*}

By now, we finish the proof.

\end{proof}

In order to estimate the right hand side of \eqref{eq: phi_2^0} and \eqref{eq: phi_2^1} and boundary term, we need the following results:
\begin{lemma}\label{lem: (pa_x phi_1, pa_y phi_1)}
For $i=0,1,$ we have that
\begin{align}\label{est: pa_x phi_1}
\int_0^t\Big(\|\pa_x(\phi_{bc,S}^i+\phi_{bc,T}^i)\|_{X^\f53}^2+\|\pa_y (\phi_{bc,S}^i+\phi_{bc,T}^i)\|_{X^\f53}^2\Big) ds\leq&\f{C}{\la^\f12}\int_0^t|h^i|_{X^{\f73}}^2ds,
\end{align}
%%%%%%%%%%%%%%%%%%%%%%%%%%%%
%\int_0^t\|\pa_y (\phi_{bc,S}^i+\phi_{bc,T}^i)\|_{X^\f53}^2ds\leq&\f{C}{\la^\f12}\int_0^t|h^i|_{X^{\f73}}^2ds,\label{est: pa_y phi_1}\\
%%%%%%%%%%%%%
\begin{align}
\int_0^t |\pa_y \phi_{bc,R}^i|_{y=0,1}|_{X^\f73}^2ds\leq& C\int_0^t\|\om_{bc,R}^i\|_{X^\f73}^2ds+\f{C}{\la^\f12}\int_0^t|h^i|_{X^\f73}^2ds,\label{est: BC-pa_y phi_2-1}\\
%%%%%%%%%%%%%%%%%%%%%
\int_0^t \Big\langle \pa_x(\phi_{bc,R}^i)_\Phi,\pa_y(\phi_{bc,R}^i)_\Phi \Big\rangle_{H^2_x}\Big|_{y=0}^{y=1}ds\leq&C\int_0^t\|\om_{bc,R}^i\|_{X^\f73}^2ds+\f{C}{\la^\f12}\int_0^t|h^i|_{X^\f73}^2ds.\label{est: BC-pa_y phi_2-2}
\end{align}

\end{lemma}
\begin{proof}
Here we only prove the case $i=0.$ The case $i=1$ is almost the same and we omit details to readers.

By Proposition \ref{pro: na_e phi} and Corollary \ref{cor: om^i-transport}, we get the \eqref{est: pa_x phi_1} proved.

Next, we deal with the boundary term. A direct calculation gives that
\begin{align*}
\pa_y\phi_{bc,R}^0|_{y=1}=&(\pa_y\widetilde{\phi}_{bc,R}^0-\pa_yg^0)|_{y=1}\\
=&\pa_y\widetilde{\phi}_{bc,R}^0|_{y=1}-(\pa_y\phi^0_{bc,S}+\pa_y\phi^0_{bc,T})|_{y=1}-(\phi^0_{bc,S}+\phi^0_{bc,T})|_{y=1},
\end{align*}
and 
\begin{align*}
\pa_y\phi_{bc,R}^0|_{y=0}=(\pa_y\widetilde{\phi}_{bc,R}^0-\pa_yg^0)|_{y=0}=\pa_y\widetilde{\phi}_{bc,R}^0|_{y=0},
\end{align*}
due to $\phi_{bc, S}^0|_{y=0}=\phi_{bc, T}^0|_{y=0}=0.$

By Corollary \ref{cor: om^i-transport}, we get
implies
\begin{align*}
\int_0^t|\pa_y\widetilde{\phi}_{bc, R}^0|_{y=0,1}|_{X^\f73}^2ds 
\leq&\f{C}{\la^\f12}\int_0^t|h^0|_{X^\f73}^2ds+C\int_0^t(\|\om_{bc, R}^0\|_{X^\f73}^2+\|\Delta_\e g^0\|_{X^\f73}^2)ds\\
\leq&C\int_0^t\|\om_{bc, R}^0\|_{X^\f73}^2ds+\f{C}{\la^\f12}\int_0^t|h^0|_{X^\f73}^2ds,
\end{align*}
where we use elliptic estimate and Calderon-Zygmund inequality 
\beno
\|\pa_y\widetilde{\phi}_{bc,R}^0\|_{X^\f73}+\|\pa_y^2\widetilde{\phi}_{bc,R}^0\|_{X^\f73}\leq C\|\om_{bc,R}^0\|_{X^\f73}+C\|\Delta_\e g^0\|_{X^\f73}.
\eeno

%
%Lemma \ref{lem: phi^i|_y=1-i}, Lemma \ref{lem: pointwise-phi^i} and Corollary \ref{cor: om^i} deduce
%\begin{align*}
%\int_0^t|\phi_{bc,S}^0|_{y=1}|_{X^\f83}^2 ds\leq& \f{C}{\la}\int_0^t|h^0|_{X^{\f73}}^2ds,\\
%\int_0^t|\pa_y\phi_{bc,S}^0|_{y=1}|_{X^\f83}^2 ds\leq& \f{C}{\la}\int_0^t|h^0|_{X^{\f73}}^2ds,\\
%\int_0^t|\phi_{bc,T}^0|_{y=1}|_{X^\f83}^2 ds\leq&\f{C}{\la^\f32}\int_0^t|h^0|_{X^{\f73}}^2ds,\\
%\int_0^t|\pa_y\phi_{bc,T}^0|_{y=1}|_{X^\f83}^2 ds\leq&\f{C}{\la^\f32}\int_0^t|h^0|_{X^{\f73}}^2ds,
%\end{align*}
%which and above estimate  follow
%\begin{align*}
%\int_0^t |\pa_y \phi_{bc,R}^0|_{y=0,1}|_{X^\f73}^2ds\leq C\int_0^t\|\om_{bc,R}^0\|_{X^\f73}^2ds+\f{C}{\la^\f12}\int_0^t|h^0|_{X^\f73}^2ds,
%\end{align*}
%and we complete \eqref{est: BC-pa_y phi_2-1}. 

For the last estimate, we use \eqref{est: BC-pa_y phi_2-1} and  \eqref{est: h-g} to imply
\begin{align*}
\int_0^t \Big\langle \pa_x(\phi_{bc,R}^0)_\Phi,\pa_y(\phi_{bc,R}^0)_\Phi \Big\rangle_{H^2_x}\Big|_{y=0}^{y=1}ds\leq& C\int_0^t|(\phi_{bc,S}^0+\phi_{bc,T}^0)|_{y=1}|_{X^\f83}|\pa_y \phi_{bc,R}^0|_{y=1}|_{X^\f73}ds\\
\leq&C\int_0^t|f^0|_{X^\f83}|\pa_y \phi_{bc,R}^0|_{y=1}|_{X^\f73}ds\\
\leq&C\int_0^t\|\om_{bc,R}^0\|_{X^\f73}^2ds+\f{C}{\la^\f12}\int_0^t|h^0|_{X^\f73}^2ds.
\end{align*}
Here, we complete this lemma.

\end{proof}

We are coming to the main part of this section. We shall give the estimate for the system \eqref{eq: phi_2^0} and \eqref{eq: phi_2^1}.
\begin{proposition}\label{pro: om_2^i}
Let $\phi_{bc,R}^0$ and $\phi_{bc,R}^1$ be the solution of \eqref{eq: phi_2^0} and \eqref{eq: phi_2^1} respectively, and $\om_{bc,R}^i=\Delta_\e \phi_{bc,R}^i$ for $i=0,1.$ Then, for every $i=0,1$, it holds that
\begin{align*}
\|\om_{bc,R}^i(t)\|_{X^{2}}^2+&\la\int_0^t  \|\om_{bc,R}^i\|_{X^{\f73}}^2ds+\int_0^t (\|\na_\e\phi_{bc,R}^i\|_{X^\f73}^2+|\pa_y\phi_{bc,R}^i|_{y=0,1}|_{X^\f73}^2)ds+\int_0^t\|\na_\e\om_{bc,R}^i\|_{X^{2}}^2ds\\
\leq&\f{C}{\la^\f12}\int_0^t|h^i|_{X^{\f73}}^2ds,
\qquad t\in[0, T],
\end{align*}
where $0<T<\min\{T_p, \f{1}{2\lambda}\}$.

\end{proposition}

\begin{proof} 
The result mainly comes from the process of Propoposition \ref{pro: om-good-1}. Here we take $(\mathcal{N}_u, \e \mathcal{N}_v)=0$ and $\e^2 f_1+f_2-C(t)\pa_x\om^p$ is replaced by $G^i=-\pa_y(\phi_{bc,S}^i+\phi_{bc,T}^i)\pa_x\om^p+\pa_x(\phi_{bc,S}^i+\phi_{bc,T}^i)\pa_y\om^p$ for $i=0,1.$ In order to estimate the source term $\int_0^t \|G^i\|_{X^\f53}^2ds$, using Lemma \ref{lem: (pa_x phi_1, pa_y phi_1)} and product estimate in Lemma \ref{lem:product-Gev}, we get
\begin{align*}
\int_0^t \|G^i\|_{X^\f53}^2ds\leq& C\int_0^t \|\pa_y(\phi_{bc,S}^i+\phi_{bc,T}^i)\|_{X^\f53}^2ds+C\int_0^t \|\pa_x(\phi_{bc,S}^i+\phi_{bc,T}^i)\|_{X^\f53}^2ds\\
\leq&\f{C}{\la^\f12}\int_0^t|h^i|_{X^{\f73}}^2ds.
\end{align*}
The only difference comes from  boundary condition $$\phi_{bc,R}^i|_{y=i}=0,\quad
\phi_{bc,R}^i|_{y=1-i}=-(\phi^i_{bc,S}+\phi^i_{bc,T})|_{y=1-i},$$ which are not zero compared with equation \eqref{eq: tri_e phi-good}. We review $T^5$ in Propoposition \ref{pro: om-good-1}. After integration by parts, the boundary term is left. More precisely, we need to estimate $\int_0^t \Big\langle \pa_x(\phi_{bc,R}^i)_\Phi,\pa_y(\phi_{bc,R}^i)_\Phi \Big\rangle_{H^2_x}\Big|_{y=0}^{y=1}ds.$ According to Lemma \ref{lem: (pa_x phi_1, pa_y phi_1)}, we have
\begin{align*}
\int_0^t \Big\langle \pa_x(\phi_{bc,R}^i)_\Phi,\pa_y(\phi_{bc,R}^i)_\Phi \Big\rangle_{H^2_x}\Big|_{y=0}^{y=1}ds\leq&C\int_0^t\|\om_{bc,R}^i\|_{X^\f73}^2ds+\f{C}{\la^\f12}\int_0^t|h^i|_{X^\f73}^2ds.
\end{align*}
Here, we take $\la$ large enough to  complete the proof.

\end{proof}

\subsection{Proof of Proposition \ref{pro: phi_{bc}}}

In this subsection, we combine all above estimates to finish the proof of Proposition  \ref{pro: phi_{bc}}. Recalling the definition of $\phi_{bc}$:
\begin{align}\label{def: phi_bc}
\phi_{bc}=\phi_{bc, S}+\phi_{bc, T}+\phi_{bc, R},
\end{align}
 we get that
\begin{align}\label{eq: phi-1-11}
\left\{
\begin{aligned}
&(\pa_t-\Delta_\e)\Delta_\e\phi_{bc}+u^p\pa_x \Delta_\e\phi_{bc}+v^p\pa_y \Delta_\e\phi_{bc}+\pa_y\phi_{bc}\pa_x \om^p-\pa_x\phi_{bc}\pa_y \om^p=0,\\
&\phi_{bc}|_{y=0,1}=0,\quad \pa_y\phi_{bc}|_{y=0}=h^0+R_{bc}^{00} +R_{bc}^{01} ,\quad \pa_y\phi_{bc}|_{y=1}=h^1+R_{bc}^{10} +R_{bc}^{11} ,\\
&\phi_{bc}|_{t=0}=0.
\end{aligned}
\right.
\end{align}
Here $R_{bc}^{ji}~ (j=0,1,~i=0,1)$ are linear operators and are defined by
\begin{align*}
R_{bc}^{00}=&\big(\pa_y\phi_{bc, T}^0  +\pa_y\phi_{bc, R}^0\big) |_{y=0},\\
R_{bc}^{01}=&\big(\pa_y\phi_{bc, S}^1 +\pa_y\phi_{bc, T}^1 +\pa_y \phi_{bc, R}^1\big) |_{y=0},\\
R_{bc}^{10}=&\big(\pa_y\phi_{bc, S}^0 +\pa_y\phi_{bc, T}^0 +\pa_y\phi_{bc, R}^0\big) |_{y=1},\\
R_{bc}^{11}=&\big(\pa_y\phi_{bc, T}^1 +\pa_y \phi_{bc, R}^1\big)|_{y=1}.
\end{align*}

Compared with the system \eqref{eq: phi-1-1}, we need to find $(h^0, h^1)$ such that
\begin{align}\label{eq:h=R}
\left\{
\begin{aligned}
h^0+R_{bc}^{00} +R_{bc}^{01} =- \pa_y\phi_{slip}|_{y=0}+C(t),\\
h^1+R_{bc}^{10} +R_{bc}^{01} =- \pa_y\phi_{slip}|_{y=1}+C(t),
\end{aligned}
\right.
\end{align}
hold. 
To do that, we define an operator $R_{bc}[h^0, h^1],$ which is defined by 
\begin{align}\label{def: R_bc}
R_{bc}[h^0, h^1]=
\left(\begin{array}{cc}R_{bc}^{00} & R_{bc}^{01} \\R_{bc}^{10} & R_{bc}^{11} \end{array}\right)
\end{align}
is a $2\times 2$ matrix operator and 
 is well-defined on the Banach space
\begin{align}\label{def: Z_bc}
Z_{bc}=\{(h^0, h^1)\in L^2(0,t; L^2)| \int_0^t|(h^0, h^1)|^2_{X^\f73}ds<+\infty\}.
\end{align}

\medskip

\begin{proposition}\label{pro: (h^0, h^1)}
There exists $\la_0\geq1$ such that if $\la\geq \la_0$, the map $R_{bc}: Z_{bc}\to Z_{bc}$ defined by \eqref{def: R_bc}  satisfies
\begin{align}\label{est: R_bc}
\int_0^t\Big|R_{bc}[h^0, h^1]\Big|_{X^\f73}^2 ds\leq&\f{C}{\la^\f12}\int_0^t|(h^0,h^1)|_{X^{\f73}}^2ds.
\end{align}
Hence, the operator $I+R_{bc}$ is invertible in $Z_{bc}$. Moreover, there exists $(h_0, h_1)\in Z_{bc}$ such that \eqref{eq:h=R} holds and $(h_0, h_1)$ is defined by
\beno
(h_0, h_1)=(I+R_{bc})^{-1}(- \pa_y\phi_{slip}|_{y=0}+C(t), - \pa_y\phi_{slip}|_{y=1}+C(t)).
\eeno
 
\end{proposition}
\begin{proof}
First, by Lemma \ref{lem: phi^i|_y=1-i}, Proposition \ref{cor: om^i-transport}, Corollary \ref{pro: om^i-transport}
 and Proposition \ref{pro: om_2^i}, it is easy to get 
\beno
\int_0^t\Big|R_{bc}[h^0, h^1]\Big|_{X^\f73}^2 ds\leq&\f{C}{\la^\f12}\int_0^t|(h^0,h^1)|_{X^{\f73}}^2ds.
\eeno
Taking $\la$ large enough, we get that the operator $I+R_{bc}$ is invertible in $Z_{bc}$. Thus, there exists $(h_0, h_1)\in Z_{bc}$ such that \eqref{eq:h=R} holds.
 
\end{proof}

Let's continue to prove Proposition \ref{pro: phi_{bc}}.   According to Proposition \ref{pro: na_e phi}, Proposition \ref{pro: Estimate for om^i-transport}, Corollary \ref{cor: om^i-transport} and Proposition \ref{pro: om_2^i}, we get by \eqref{def: phi_bc} that
\begin{align*}
\int_0^t \|\na_\e\phi_{bc} \|_{X^\f73}^2ds\leq &\int_0^t \|\na_\e\phi_{bc,S} \|_{X^\f52}^2ds+\int_0^t \|\na_\e\phi_{bc,T} \|_{X^\f52}^2ds+\int_0^t \|\na_\e\phi_{bc,R} \|_{X^\f73}^2ds\\
\leq&\f{C}{\la^\f12}\int_0^t|(h^0, h^1)|_{X^\f73}^2ds,
\end{align*}
and
\begin{align*}
\int_0^t \|\varphi \Delta_\e \phi_{bc}\|^2_{X^2}ds\leq& \int_0^t \|\varphi \Delta_\e \phi_{bc,S}\|^2_{X^\f 52}ds+\int_0^t \|\varphi \Delta_\e \phi_{bc,T}\|^2_{X^\f52}ds+\int_0^t \| \Delta_\e \phi_{bc,R}\|^2_{X^\f73}ds\\
\leq&\f{C}{\la^\f12}\int_0^t|(h^0, h^1)|_{X^\f73}^2ds,
\end{align*}
which imply
\begin{align*}
\int_0^t \|\na_\e\phi_{bc} \|_{X^\f73}^2+\|\varphi \Delta_\e \phi_{bc}\|^2_{X^2}ds\leq&\f{C}{\la^\f12}\int_0^t|(h^0, h^1)|_{X^\f73}^2ds.
\end{align*}

Due to Proposition \ref{pro: (h^0, h^1)} and taking $\mathcal{A}=(I+R_{bc})^{-1},$ we know $\mathcal{A}$ is a zero-order bounded operator in $Z_{bc}$ and  obtain
\begin{align*}
\int_0^t|(h^0, h^1)|_{X^\f73}^2ds=&\int_0^t|\mathcal{A}(- \pa_y\phi_{slip}|_{y=0}+C(s), - \pa_y\phi_{slip}|_{y=1}+C(s))|_{X^\f73}^2ds\\
\leq&C\int_0^t \Big(|\na_\e\phi_{slip}|_{y=0,1}|^2_{X^\f73}+|C(s)|^2\Big)ds,
\end{align*}
which finish this proposition.
%\begin{proposition}\label{pro: (h^0, h^1)}
%There exists $\la_0\geq1$ such that if $\la\geq \la_0$, the map $Z_{bc}\to Z_{bc}$ defined by \eqref{def: Z_bc}  satisfies
%\begin{align}\label{est: R_bc}
%\int_0^t\Big|R_{bc}[h^0, h^1]\Big|_{X^\f73}^2 ds\leq&\f{C}{\la^\f12}\int_0^t|(h^0,h^1)|_{X^{\f73}}^2ds.
%\end{align}
%Hence, the operator $I+R_{bc}$ is invertible in $Z_{bc}$ where $R_{bc}$ is given in \eqref{def: R_bc}, and the map
%\begin{align*}
%\phi_{bc}[h^0, h^1]:=\phi_{app}[(I+R_{bc})^{-1}[h^0, h^1]],
%\end{align*}
%gives the solution to \eqref{eq: phi-1} and satisfies
%\begin{align}\label{est: na_e phi_app}
%\int_0^t \|\na_\e\phi_{bc}[h^0, h^1]\|_{X^\f73}^2ds\leq& C\int_0^t|(h^0, h^1)|_{X^\f73}^2 ds,
%\end{align}
%where $C$ is a universal constant.
%\end{proposition}
%\begin{proof}
%Estimate \eqref{est: R_bc} is obtained by combining Lemma \ref{lem: phi^i|_y=1-i}, Corollary \ref{cor: om^i} and Proposition \ref{pro: om_2^i} together. In particular, we have
%\begin{align}
%\int_0^t |(I+R_{bc})^{-1}[h^0, h^1]|_{X^\f73}^2ds\leq C\int_0^t |(h^0, h^1)|_{X^\f73}^2ds.
%\end{align}
%Then Proposition \ref{pro: na_e phi}, Corollary \ref{cor: om^i} and Proposition \ref{pro: om_2^i} give \eqref{est: na_e phi_app}. The proof is complete.
%\end{proof}

\bigskip

\section* {Acknowledgments}
The authors would like to thank Professors Zhifei Zhang for the valuable discussions and suggestions. C. Wang is partially supported by NSF of China under Grant 12071008. Y. Wang is partially supported by NSF of China under Grant 12101431. 

%Z. Zhang is partially supported by NSF of China under Grant 12171010.  

% Part of this work was done when Yuxi Wang was visiting Peking University in July 2021. Yuxi Wang is partially supported by Sino-Russian Mathematics Center. 

\end{document}